\documentclass[12pt,leqno]{amsart}
\usepackage{amsmath,amssymb,amsfonts}
\usepackage{eucal,enumerate,graphicx}

\newtheorem{lem}{Lemma}[section]
\newtheorem{cor}[lem]{Corollary}
\newtheorem{prop}[lem]{Proposition}
\newtheorem{thm}[lem]{Theorem}

  \newtheorem{exama}[lem]{Example}
\newenvironment{exam}{\begin{exama}\rm}{\end{exama}}
\numberwithin{equation}{section}

\renewcommand{\phi}{\varphi}                 % Personal preferences.
\renewcommand{\epsilon}{\varepsilon}
\newcommand\eps{\varepsilon}
\newcommand\eset{\varnothing}
\newcommand\setm{\setminus}

\newcommand\cupdot {\mbox{\hspace{.15em}$\cup$\hspace{-.47em}$\cdot$\hspace{.4em}}} 
\newcommand\inv{^{-1}}
\newcommand\transpose{^{\text{\rm T}}}
\newcommand\textb{\text{\rm b}}
\newcommand\pib{\pi_{\textb}}

\newcommand\Lat{\operatorname{Lat}}
\newcommand\clos{\operatorname{clos}}
\newcommand\bcl{\operatorname{bcl}}
\newcommand\rk{\operatorname{rk}}

\renewcommand\Im{\operatorname{Im}}           % \Im is predefined as fraktur I.
\newcommand\full{^\bullet}
\newcommand\Char{\operatorname{char}}
\newcommand\CP{\operatorname{CP}}

\newcommand\Beta{{\mathrm B}}
\newcommand\Eta{{\mathrm H}}

\newcommand\bv{\mathbf{v}}

\newcommand\bx{\mathbf{x}}
\newcommand\by{\mathbf{y}}

\newcommand\cA{\mathcal{A}}
\newcommand\cB{\mathcal{B}}
\newcommand\cC{\mathcal{C}}
\newcommand\cD{\mathcal{D}}
\newcommand\cH{\mathcal{H}}
\newcommand\cL{\mathcal{L}}
\newcommand\cP{\mathcal{P}}
\newcommand\cS{\mathcal{S}}

\newcommand\bbR{\mathbb{R}}
\newcommand\bbZ{\mathbb{Z}}

\newcommand\F{\mathbf{F}}
\newcommand\B{\mathbf{b}}
\newcommand\brho{\boldsymbol\rho}

\begin{document}

\begin{center}
{\sc 
{\Large 
Signed Graphs and Geometry}
}

\bigskip
{\sc 
{\large Thomas Zaslavsky} 
}
\\[10pt]
Department of Mathematical Sciences \\
Binghamton University \\
Binghamton, NY 13902-6000, U.S.A.\\
E-mail:  {\tt zaslav@math.binghamton.edu}

\end{center}

\bigskip

\pagestyle{myheadings}\markboth{{\sc 
Signed Graphs and Geometry% $|$ 
}}{{\sc Thomas Zaslavsky}}%{{\sc IWSSG-2011 Manathavady $|$ 2--6 September 2011}} 

\vfill

%%===================================================%%

\section*{Introduction}

These lecture notes are a personal introduction to signed graphs, concentrating on the aspects that have been most persistently interesting to me.  They are just a few corners of signed graph theory; I am leaving out a great deal.  The emphasis is on the way signed graphs arise naturally from geometry, especially from the geometry of the classical root systems.  Most of the properties I discuss generalize those of unsigned graphs, but the constructions and proofs are often more complicated.  
My aim is a coherent presentation of the subject, with a few illustrative proofs and adequate references.  Hence the arrangement of the notes is topical with only occasional remarks about the historical course of development.  
Though this is mainly an expository survey, some of the results have not hitherto been published.

For a fairly comprehensive list of articles on signed graphs, generalizations, and related work see \cite{BSG}; for (much of the) terminology see \cite{Gloss}.  
The principal reference for most of the more elementary properties of signed graphs treated here is Zaslavsky (1982a).  A simple introduction to the hyperplane geometry is Zaslavsky (1981a).  
(\emph{N.B.}  Citations in the style Name (YEARa) refer to author Name's item (YEARa) in \cite{BSG}.)  
Many of my articles can be downloaded from my Web site, 
\begin{center}
{\tt http://www.math.binghamton.edu/zaslav/Tpapers/}
\end{center}

Now, bon voyage!  
\emph{Shubha Yatra!}

\vfill
%%===================================================%%
\section{Graphs}\label{g}

We begin with a review of graph theory.  Much is familiar but signed graphs require several extensions of the ordinary theory.

A graph is $\Gamma = (V,E)$, where $V := V(\Gamma)$ is the \emph{vertex set} and $E := E(\Gamma)$ is the \emph{edge set}.  
In these lectures all graphs are finite.

\subsubsection*{Notation}
\begin{itemize}
\item The \emph{set sum} or \emph{symmetric difference} of two sets $A$ and $B$ is denoted by $$A \oplus B := (A \setm B) \cup (B \setm A).$$
\item $n := |V|$, called the \emph{order} of $\Gamma$.
\item $V(e)$ is the multiset of vertices of the edge $e$.
\item If $S \subseteq E$, $V(S)$ is the set of endpoints of edges in $S$.
\item If $S \subseteq E$, its \emph{complement} is $S^c := E \setm S$.
\item If $X \subseteq V$, its \emph{complement} is $X^c := V \setm X$.
\end{itemize}

\subsubsection*{Edges and edge sets}
\begin{itemize}
\item We allow multiple edges as well as loops and oddball objects called half and loose edges.
\item There are four kinds of edge:  \\[6pt]
A \emph{link} has two distinct endpoints.    	\hfill\includegraphics{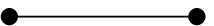}\hbox to 2in{}\\[6pt]
A \emph{loop} has two equal endpoints.   	\hfill\raisebox{-3pt}{\includegraphics{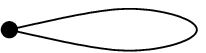}}\hbox to 2in{}\\[6pt]
A \emph{half edge} has one endpoint.      		\hfill\raisebox{-2pt}{\includegraphics{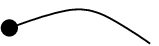}}\hbox to 2.2in{}\\[5pt]
A \emph{loose edge} has no endpoints.    	\hfill\includegraphics{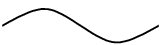}\hbox to 2.15in{}\medskip
\item An \emph{ordinary edge} is a link or a loop.  
\item The set of loose edges of $\Gamma$ is $E_0(\Gamma)$.  The set of ordinary edges of $\Gamma$ is $E_* := E_*(\Gamma)$.  
\item Edges are \emph{parallel} if they have the same endpoints.   %\quad  \raisebox{-2pt}{\includegraphics{F-parallel}}	%TOO TRIVIAL
\medskip
\item An \emph{ordinary graph} is a graph in which every edge is a link or a loop. 
\item A \emph{link graph} is a graph whose edges are links.  
\item A \emph{simple graph} is a link graph with no parallel edges.
\medskip
\item %\noindent\parbox{4in}{
$E(X,Y)$, where $X,Y \subseteq V$, is the set of edges with one endpoint in $X$ and the other in $Y$.  (Every such edge must be a link or, if $X \cap Y \neq \eset$, a loop.)
A \emph{cut} or \emph{cutset} is an edge set $E(X,X^c)$ that is nonempty.
\end{itemize}

\subsubsection*{Vertices and vertex sets in $\Gamma$}  Let $X \subseteq V$.
\begin{itemize}
\item An \emph{isolated vertex} is a vertex that has no incident edges; i.e., a vertex of degree $0$.
\item $X$ is \emph{stable} or \emph{independent} if, aside from loose edges, no edge has all endpoints in $X$.
\end{itemize}

\pagebreak[3]
\subsubsection*{Degrees and regularity}
\begin{itemize}
\item The \emph{degree} of a vertex $v$, $d(v):=d_\Gamma(v)$, is the number of edge ends of which $v$ is an endpoint.  A loop counts twice, once for each end. 
\item $\Gamma$ is \emph{regular} if every vertex has the same degree.  If that degree is $k$, it is \emph{$k$-regular}.
\end{itemize}

\subsubsection*{Walks, trails, paths, circles}
\begin{itemize}
\item A \emph{walk} is a sequence $v_0 e_1 v_1 \cdots e_{l} v_l$ where $V(e_i) = \{ v_{i-1}, v_i \}$ and $l \geq 0$.  Its \emph{length} is $l$.  A walk may be written $e_1 e_2 \cdots e_l$ or $v_0v_1\cdots v_l$.
\item A \emph{closed walk} is a walk where $l \geq 1$ and $v_0 = v_l$.
\item A \emph{trail} is a walk with no repeated edges.
\item A \emph{path} or \emph{open path} is a trail with no repeated vertex, or the graph of such a trail (technically, the latter is a \emph{path graph}), or the edge set of a path graph.
\item A \emph{closed path} is a closed trail with no repeated vertex other than that $v_0=v_l$.  (Thus, a closed path is not a path.)
\item A \emph{circle} (also called `cycle', `polygon', etc.) is the graph, or the edge set, of a closed path.  Equivalently, it is a connected, regular graph with degree 2, or its edge set.  %A circle of length 2 is a \emph{digon}, one of length 3 is a \emph{triangle}, etc.
\item $\cC = \cC(\Gamma)$ is the class of all circles in $\Gamma$.
\end{itemize}

\subsubsection*{Examples}
\begin{itemize}
\item $K_n$ is the complete graph of order $n$.  $K_X$ is the complete graph with vertex set $X$.
\item $K_n^c$ is the edgeless graph of order $n$.
\item $\Gamma^c$ is the complement of $\Gamma$, if $\Gamma$ is simple.
\item $P_l$ is a path of length $l$ (as a graph or edge set).
\item $C_l$ is a circle of length $l$ (as a graph or edge set).
\item $K_{r,s}$ is the complete bipartite graph with $r$ left vertices and $s$ right vertices.  $K_{X,Y}$ is the complete bipartite graph with left vertex set $X$ and right vertex set $Y$.
\item The empty graph, $\eset := (\eset,\eset)$, has no vertices and no edges.  It is not connected.
\end{itemize}

\subsubsection*{Types of subgraph}  In $\Gamma$, let $X\subseteq V$ and $S \subseteq E$.
\begin{itemize}
\item A \emph{component} (or \emph{connected component}) of $\Gamma$ is a maximal connected subgraph, excluding loose edges.  An \emph{isolated vertex} is a component that has one vertex and no edges.  A loose edge is not a component.
\item $c(\Gamma)$ is the number of components of $\Gamma$.  $c(S)$ is short for $c(V,S)$.
\item A \emph{spanning subgraph} is $\Gamma' \subseteq \Gamma$ such that $V' = V$.
\item $\Gamma|S := (V,S)$.  This is a spanning subgraph.
\item  $S{:}X := \{e \in S : \eset \neq V(e) \subseteq X \} = (E{:}X) \cap S$.  We often write $S{:}X$ as short for the subgraph $(X, S{:}X)$.
\item The \emph{induced subgraph} $\Gamma{:}X$ is the subgraph $\Gamma{:}X := (X, E{:}X).$  An induced subgraph has no loose edges.  
We often write $E{:}X$ as short for $(X, E{:}X)$.  
\item $\Gamma \setm S := (V, E \setm S) = \Gamma|S^c$.
\item $\Gamma \setm X$ is the subgraph with 
$$
V(\Gamma \setm X) := X^c \text{ and } E(\Gamma \setm X) := \{e \in E \mid V(e) \subseteq V \setm X \}.
$$
We say $X$ is \emph{deleted} from $\Gamma$.  $\Gamma \setm X$ includes all loose edges, if there are any (unlike $\Gamma{:}X^c$, which has no loose edges).
\end{itemize}

\subsubsection*{Graph structures and types}
\begin{itemize}
\item \raisebox{-9pt}{\parbox{4in}{A \emph{theta graph} is the union of three internally disjoint paths that have the same endpoints.\medskip}
\hspace{.5in}\raisebox{-9pt}{\includegraphics[scale=.9]{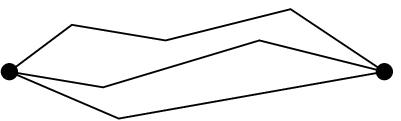}}}
\item A \emph{block} of $\Gamma$ is a maximal subgraph without loose edges, such that every pair of edges is in a circle together.  The simplest kinds of block are an isolated vertex, and a subgraph $(\{v\},\{e\})$ where $e$ is a loop or half edge at vertex $v$.  A loose edge is not in any block of $\Gamma$.
\item $\Gamma$ is \emph{inseparable} if it has only one block.
\item A \emph{cutpoint} is a vertex that belongs to more than one block.
\end{itemize}

\subsubsection*{Fundamental circles}\

\noindent\parbox{4in}{\quad
Let $T$ be a maximal forest in $\Gamma$.  If $e \in E_* \setm T$, there is a unique circle $C_e \subseteq T \cup \{e\}$.  The \emph{fundamental system of circles} for $\Gamma$, with respect to $T$, is the set of all circles $C_e$ for $e \in E_* \setm T$.  
}
\hfill\raisebox{-1cm}{\includegraphics[scale=.8]{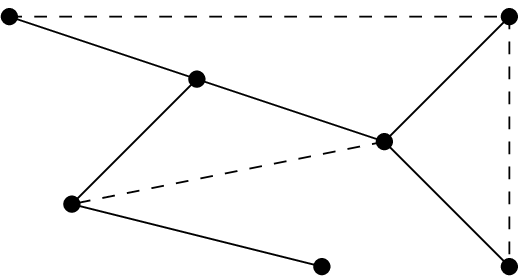}}
\medskip
\begin{prop}\label{P:fundsystem}
Choose a maximal forest $T$.  Every circle in $\Gamma$ is the set sum of fundamental circles with respect to $T$.
\end{prop}

\begin{proof}
$C =  \bigoplus_{e \in C \setm T} C_{T}(e)$.
\end{proof}

%%===================================================%%
\section{Signed Graphs}\label{sg}

A \emph{signed graph} $\Sigma = (\Gamma,\sigma) = (V,E,\sigma)$ is a graph $\Gamma = (V,E)$ together with a function $\sigma$ that assigns a sign, $\sigma(e) \in \{+,-\}$, to each ordinary edge (link or loop) in $\Gamma$.  $\sigma$ is called the \emph{signature} (or \emph{sign function}).  A half or loose edge does not get a sign.  Thus, the signature is $\sigma: E_* \to \{+,-\}$.

Notation:
\begin{itemize}
\item $|\Sigma|$ is the \emph{underlying graph} $\Gamma$.
\item $E^+ := \sigma\inv(+) = \{ e \in E : \sigma(e)=+ \}.$  The \emph{positive subgraph} is $\Sigma^+ := (V,E^+).$
\item $E^- := \sigma\inv(+) = \{ e \in E : \sigma(e)=- \}.$  The \emph{negative subgraph} is $\Sigma^- := (V,E^-).$
\item $+\Gamma := (\Gamma,+)$ is an \emph{all-positive} signed graph (every ordinary edge is $+$).  $e \in E_*(\Gamma)$ becomes $+e \in +E = E(+\Gamma)$.
\item $-\Gamma := (\Gamma,-)$ is an \emph{all-negative} signed graph (every ordinary edge is $-$).  $e \in E_*(\Gamma)$ becomes $-e \in -E = E(-\Gamma)$.
\item $\pm\Gamma := (+\Gamma) \cup (-\Gamma)$.  $E(\pm\Gamma) = \pm E := (+E) \cup (-E).$  This is the \emph{signed expansion} of $\Gamma$.
\item $\Sigma\full := \Sigma$ with a half edge or negative loop attached to every vertex that does not have one.  $\Sigma\full$ is called a \emph{full} signed graph.
\item $\Sigma^\circ := \Sigma$ with a negative loop attached to every vertex that does not have one.
\item Equivalent notations for the sign group are $\{+,-\}$, $\{+1,-1\}$, etc.  We consider $+$ and $+1$ equivalent, also $-$ and $-1$; this is important when we add signs in matrix matters.  (Another notation is $\bbZ_2 := \{0,1\}$ modulo 2, but that is unsuitable here because we need addition to imply summation of $\pm1$'s.)
\end{itemize}

Signed graphs $\Sigma_1$ and $\Sigma_2$ are \emph{isomorphic}, written $\Sigma_1 \cong \Sigma_2$, if there is an isomorphism between their underlying graphs that preserves the signs of edges.

%%%%%%%%%%%%

%%%%%
\subsection{Balance}\label{bal}\

Balance or imbalance is the fundamental property of a signed graph.

\subsubsection{Signs and balance}\

\begin{itemize}
\item The \emph{sign of a walk} $W$, $\sigma(W)$, is the product of the signs of its edges, including repeated edges.
\item The \emph{sign of an edge set} $S$, $\sigma(S)$, is the product of the signs of its edges, without repetition.
\item The sign of a circle $C$, $\sigma(C)$, is the same whether the circle is treated as a walk or as an edge set.  
\item The \emph{class of positive circles} is 
$$\cB = \cB(\Sigma) := \{ C \in \cC(|\Sigma|) : \sigma(C) = + \}.$$
\item $\Sigma$ is \emph{balanced} if it has no half edges and every circle in it is positive.  Similarly, any subgraph or edge set is balanced if it has no half edges and every circle in it is positive.
\item A circle is balanced if and only if it is positive.  However, in general, a walk cannot be balanced because it is not a graph or edge set.
\item A \emph{negative digon} is a circle of length 2 (i.e., a pair of parallel edges) that has one positive edge and one negative edge.
\item $b(\Sigma)$ is the number of components of $\Sigma$ (omitting loose edges) that are balanced.  $b(S)$ is short for $b(\Sigma|S)$.
\item $\pib(\Sigma) := \{ V(\Sigma') : \Sigma' \text{ is a balanced component of } \Sigma\}.$  Then $b(\Sigma) = |\pib(\Sigma)|.$  $\pib(S)$ is short for $\pib(\Sigma|S)$.  %If $w\in V$ and $w \in W \in \pib(S)$, then $[w] := W$.
\item $V_0(\Sigma)$ is the set of vertices of unbalanced components of $\Sigma$.  Formally, $V_0(\Sigma) := V \setm \bigcup_{W \in \pib(\Sigma)} W.$  $V_0(S)$ is short for $V_0(\Sigma|S)$.
\end{itemize}
\smallskip

\noindent\parbox{2.5in}{\quad In the example at the right, 
$\pib(\Sigma) = \{B_1, B_2\}$ and $V_0(\Sigma) = V \setm (B_1 \cup B_2)$.}
\hfill
\raisebox{-2cm}{\includegraphics[scale=.8]{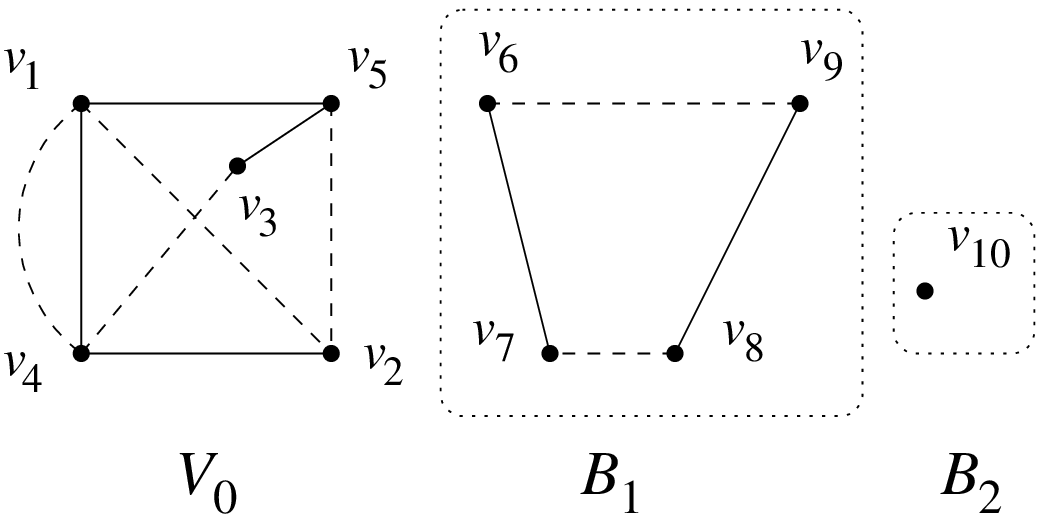}}	
\bigskip

\subsubsection{Criteria for balance}\

A \emph{bipartition} of a set $X$ is an unordered pair $\{ X_1, X_2 \}$ such that $X_1 \cup X_2 = X$ and $X_1 \cap X_2 = \emptyset$.  $X_1$ or $X_2$ could be empty.

\begin{thm}[{Harary's Balance Theorem (1953a)}] \label{T:balance}
The following statements about a signed graph are equivalent.
\begin{enumerate}[{\rm(i)}]
\item $\Sigma$ is balanced.
\label{T:bal}
\item $\Sigma$ has no half edges and there is a bipartition $V = V_1 \cupdot V_2$ such that $E^- = E(V_1,V_2)$.
\label{T:balbipartition}
\item $\Sigma$ has no half edges and any two paths with the same endpoints have the same sign.
\label{T:balpath}
\end{enumerate}
\end{thm}

We give a short proof (not the original one) after Corollary \ref{C:swbalance}.
\medskip

\noindent\parbox{3in}{\quad I like to call $\{V_1, V_2\}$ as in the theorem a \emph{Harary bipartition} of $\Sigma$.  An example appears at the right.  The circled and boxed vertices form the two sets of the bipartition.}
\hfill%\hspace{.5in}
\raisebox{-1.5cm}{\includegraphics[scale=.9]{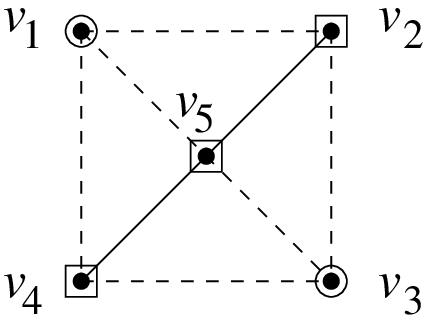}}
\ \ 
\parbox{7em}{$V_1 = \{v_1,v_3\}, \\ 
V_2 = \{v_2,v_4,v_5\}$}

\begin{cor}\label{C:balbipartite}
$-\Gamma$ is balanced if and only if $\Gamma$ is bipartite.
\end{cor}

Thus, balance is a generalization of bipartiteness.

Balance is determined by blocks.

\begin{prop} \label{P:balblocks}
$\Sigma$ is balanced if and only if every block is balanced.
\end{prop}

Deciding whether a signed graph is balanced or not is easy; see the algorithm in Section \ref{balalg}.

\subsubsection{Balancing vertices and edges}\

\begin{itemize}
\item A \emph{balancing vertex} is a vertex $v$ such that $\Sigma \setm v$ is balanced although $\Sigma$ is unbalanced.
\item A \emph{partial balancing edge} is an edge $e$ such that $\Sigma \setm e$ has more balanced components than does $\Sigma$.
\item A \emph{total balancing edge} is an edge $e$ such that $\Sigma \setm e$ is balanced although $\Sigma$ is not balanced.  A total balancing edge is a partial balancing edge, but a partial balancing edge may not be a total balancing edge.
\end{itemize}

Partial balancing edges are closely related to the geometry of signed graphs, so we record a classification.

\begin{prop}\label{P:baledge}
An edge $e$ is a partial balancing edge of\/ $\Sigma$ if and only if it is either 
\begin{enumerate}[{\rm(a)}]
\item an isthmus between two components of\/ $\Sigma \setm e$, of which at least one is balanced, or
\label{P:balisthmus}
\item a negative loop or half edge in a component\/ $\Sigma'$ such that\/ $\Sigma' \setm e$ is balanced, or
\label{P:balloop}
\item a link with endpoints $v,w$, which is not an isthmus, in a component\/ $\Sigma'$ such that $\Sigma' \setm e$ is balanced and every $vw$-path in $\Sigma' \setm e$ has sign $-\sigma(e)$.
\label{P:ballink}
\end{enumerate}
\end{prop}

\begin{proof}
We may assume $\Sigma$ is connected.

\eqref{P:balisthmus}  It is easy to see when an isthmus is a partial balancing edge.

\eqref{P:balloop}  It is also easy to see when a negative loop or a half edge is a partial balancing edge.

\eqref{P:ballink}  Suppose $e$ is a link $vw$, not an isthmus.  For $e$ to be a partial balancing edge it is necessary that $\Sigma \setm e$ be balanced.  Supposing that is true, $e$ is a partial balancing edge if and only if $\Sigma$ is unbalanced, which is true if and only if $\sigma(e)$ differs from the sign of some $vw$-path $P$.  By Theorem \ref{T:balance}\eqref{T:balpath}, the choice of $P$ does not matter.
\end{proof}

In the next diagram, `b' denotes a partial balancing edge, of which there are several.
\bigskip
\begin{center}
\includegraphics[scale=.6]{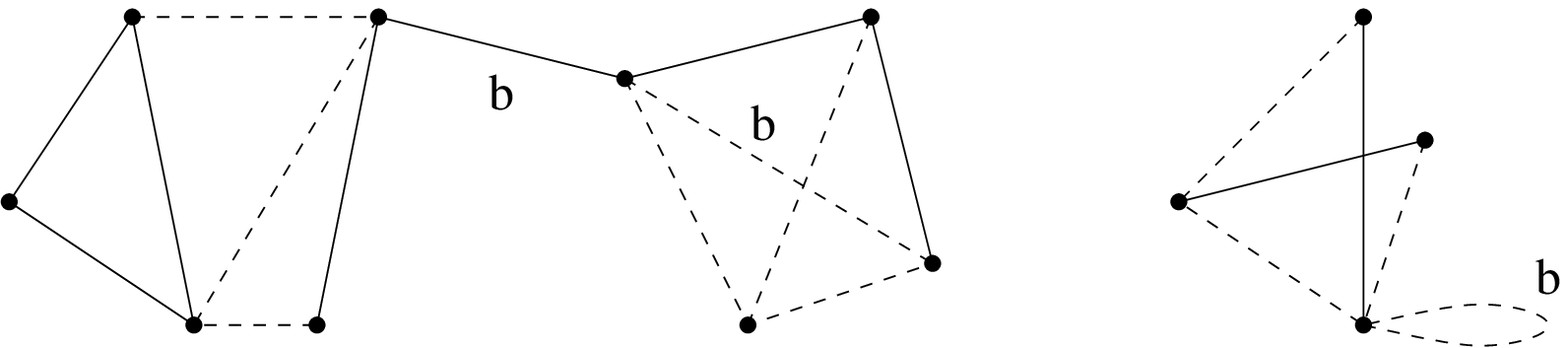}	% pictures of the three types of balancing edge.
\end{center}

Determining whether $\Sigma$ has a partial balancing edge, and finding them all, is easy.  One tests each component $\Sigma'$ for balance.  If it is unbalanced, test each edge $e$ to see whether every component of $\Sigma' \setm e$ is balanced.  That is not the most efficient method; the algorithm in Section \ref{balalg} for determining balance is adaptable to testing many edges at once.  I do not recall any explicit published algorithms for this problem.

\subsubsection{Balancing edge sets}\ %NOT USED.

\begin{itemize}
\item \emph{Partial balancing set}: $S$ such that $b(\Sigma \setm S) > b(\Sigma)$.
\item \emph{Total balancing set}: $S$ such that $\Sigma \setm S$ is balanced but $\Sigma$ is not balanced.
\end{itemize}
\medskip

Determining the minimum size of a partial or total balancing set is an NP-hard problem.  It includes the known NP-hard problem of determining the maximum cut size in a graph, because a minimum partial balancing set in $-\Gamma$ is the complement of a maximum cutset in $\Gamma$.

%%%%%
\subsection{Switching}\label{sw}\

A \emph{switching function} for $\Sigma$ is a function $\zeta: V \to \{ +, - \}$.  The \emph{switched signature} is $\sigma^\zeta(e) := \zeta(v) \sigma(e) \zeta(w)$, where $e$ has endpoints $v,w$.  The \emph{switched signed graph} is $\Sigma^\zeta := (|\Sigma|, \sigma^\zeta)$.  We say $\Sigma$ is \emph{switched by $\zeta$}.  Note that $\Sigma^\zeta = \Sigma^{-\zeta}$.

If $X \subseteq V$, \emph{switching $\Sigma$ by $X$} (or simply \emph{switching $X$}) means reversing the sign of every edge in the cutset $E(X,X^c)$.  The switched graph is $\Sigma^X$.  This is the same as $\Sigma^\zeta$ where $\zeta(v) := -$ if and only if $v \in X$.  Switching by $\zeta$ or $X$ is the same operation with different notation.  Note that $\Sigma^X = \Sigma^{X^c}$.

Switching a one-vertex set $\{v_1\}$: \hfill\raisebox{-1.5cm}{\includegraphics[scale=.8]{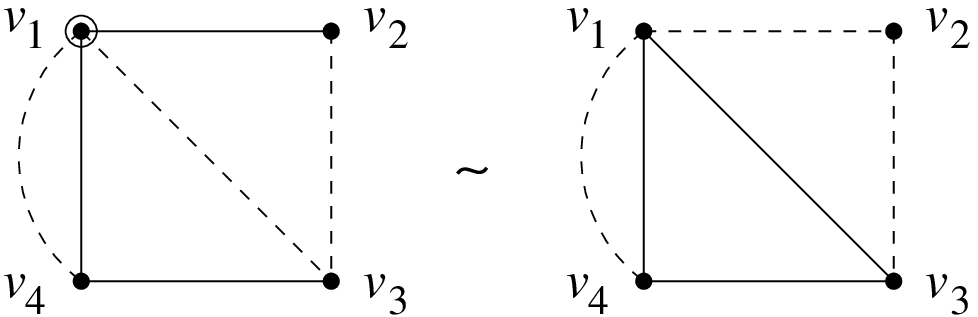}}

\pagebreak[3]

\begin{prop}\label{P:switchingequiv}\ 
\begin{enumerate}[{\rm(a)}]
\item Switching leaves the signs of all closed walks, including all circles, unchanged.  Thus, $\cB(\Sigma^{\zeta}) = \cB(\Sigma)$.
\label{P:swwalk}
\item If $|\Sigma_1| = |\Sigma_2|$ and $\cB(\Sigma_1) = \cB(\Sigma_2)$, then there exists a switching function $\zeta$ such that $\Sigma_2 = \Sigma_1 ^{\zeta}$.
\label{P:swB}
\end{enumerate}
\end{prop}

\begin{proof}[Proof of \eqref{P:swwalk} by formula]
Let $\zeta$ be a switching function and let $W = v_0 e_0 v_1 e_1 v_ 2 \cdots v_{n-1} e_{n-1} v_0$ be a closed walk.  Then 
\begin{align*}%\label{E:}
\sigma^{\zeta}(W) &=  \big[ \zeta(v_0) \sigma(e_0) \zeta(v_1) \big] \big[ \zeta(v_1) \sigma(e_1) \zeta(v_2) \big] \dots \big[ \zeta (v_{n-1}) \sigma(e_{n-1}) \zeta (v_0) \big] \\
&= \sigma(e_0)\sigma(e_1) \cdots \sigma(e_{n-1}) = \sigma(W).
\qedhere
\end{align*}
\end{proof}

\begin{proof}[Proof of \eqref{P:swB} by defining a switching function]
We may assume $\Sigma_1$ is connected.  Pick a spanning tree $T$ and a vertex $v_0$.  Define 
$$
\zeta(v) := \sigma_1(T_{v_0v})\sigma_2(T_{v_0v})
$$
where $T_{v_0v}$ is the path in $T$ from $v_0$ to $v$.  
Now it is easy to calculate that $\Sigma_1^\zeta = \Sigma_2$.
\end{proof}

Signed graphs $\Sigma_1$ and $\Sigma_2$ are \emph{switching equivalent}, written $\Sigma_1 \sim \Sigma_2$, if they have the same underlying graph and there exists a switching function $\zeta$ such that $\Sigma_1^\zeta \cong \Sigma_2$.  The equivalence class of $\Sigma$, 
$$
[\Sigma] := \{ \Sigma' : \Sigma' \sim \Sigma \},
$$ 
is called its \emph{switching class}.

Similarly, $\Sigma_1$ and $\Sigma_2$ are \emph{switching isomorphic}, written $\Sigma_1 \simeq \Sigma_2$, if $\Sigma_1$ is isomorphic to a switching of $\Sigma_2$.  The equivalence class of $\Sigma$ is called its \emph{switching isomorphism class}.

In the next figure not all signed graphs are switching equivalent: $\Sigma_2 \sim \Sigma_3$ but $\Sigma_1 \not\sim \Sigma_2 , \Sigma_3$. 
However, all are switching isomorphic: $\Sigma_1 \simeq \Sigma_2 \simeq \Sigma_3$.
\medskip
\begin{center}
\includegraphics[scale=.8]{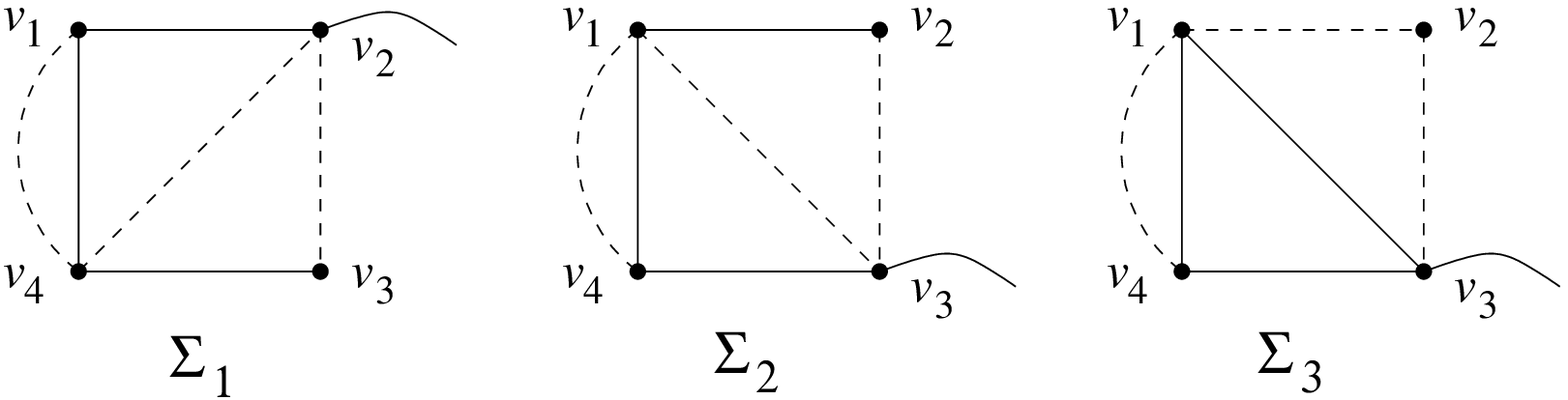}	% Sw. equiv. & sw. isom. graphs.
\end{center}

\begin{prop}\label{P:swequivalence}
Switching equivalence, $\sim$, is an equivalence relation on signatures of a fixed underlying graph.  

Switching isomorphism, $\simeq$, is an equivalence relation on signed graphs.
\end{prop}

\begin{proof}
Obvious!
\end{proof}

\subsubsection{Switching and balance}\

\begin{cor}\label{C:swbalance}
$\Sigma$ is balanced if and only if it has no half edges and it is switching equivalent to $+|\Sigma|$.
\end{cor}

\begin{proof}[Short Proof of Harary's Balance Theorem \ref{T:balance}]
We may assume $\Sigma$ has no half edges.

$\Sigma$ is balanced $\iff$ (by Proposition \ref{P:switchingequiv}) it is a switching of $+|\Sigma|$ $\iff$ it equals $(+|\Sigma|)^{V_1}$ for some $V_1 \subseteq V$.  

If $\Sigma$ is balanced, $\{V_1, V \setm V_2\}$ is the required bipartition.  Conversely, if $\Sigma$ has a Harary bipartition $\{V_1, V \setm V_2\}$, then $\Sigma = (+|\Sigma|)^{V_1}$.  That proves \eqref{T:bal} $\iff$ \eqref{T:balbipartition}.

Suppose $P, Q$ are two $vw$-paths.  Then $PQ\inv$ is a closed walk, whose sign is $+$ $\iff$ $\sigma(P) = \sigma(Q)$.  As switching does not alter the sign of closed walks, $\sigma(P) = \sigma(Q)$ when $\Sigma = (+|\Sigma|)^{V_1}$.  Thus, balance implies \eqref{T:balpath}.  To prove the inverse, suppose $\Sigma$ is not balanced; then it has a negative circle.  As the circle sign is unaffected by switching, $\Sigma$ cannot switch to $+|\Sigma|$.
\end{proof}

The original proof was much less efficient---as is often the case with original proofs.

\subsubsection{An algorithm to detect balance}\label{balalg}\

Assume $\Sigma$ is connected.  We can apply the proof of Proposition \ref{P:switchingequiv}(ii) to determine whether $\Sigma$ can be switched to all positive.  That is:
\begin{enumerate}
\item Choose a spanning tree $T$ and a root vertex $v_0$.
\item Calculate the function $\zeta(v) = \sigma(T_{v_0v})$ of that proof.
\item Switch by $\zeta$.
\item Look for negative non-tree edges.  $\Sigma$ is balanced $\iff$ all non-tree edges are positive (and none is a half edge).
\end{enumerate}

This is essentially the fast algorithm of Hansen (1978a) (his Algorithm 1) and Harary and Kabell (1980a).  The use of switching makes it perhaps more obvious than in the original publications.

%%%%%%%%%%%%
\subsection{Deletion, contraction, and minors}\label{sg.minors}\

Here $R, S$ denote subsets of $E$.  A \emph{component of $S$} means a component of $(V,S)$.

The \emph{deletion} of $S$ (or, the \emph{deletion of\/ $\Sigma$ by $S$}) is the signed graph $\Sigma \setm S := (V,S^c,\sigma|_{S^c})$.

The \emph{contraction} of $S$ (or, the \emph{contraction of\/ $\Sigma$ by $S$}) is a signed graph $\Sigma / S$, to be defined next, in two stages.  Contracting a single edge at a time is simpler to describe, but contracting an entire set of edges is more useful for further developments.

%%%%%%
\subsubsection{Contracting an edge $e$}\label{sg.contractedge}\

If $e$ is a positive link, delete $e$ and identify its endpoints; do not change any edge signs.   (This is the same as contracting a link in an unsigned graph.)

If $e$ is a negative link, switch $\Sigma$ by a switching function $\zeta$, chosen so $e$ is positive in $\Sigma^\zeta$; then contract $e$ as a positive link.  The choice of $\zeta$ does not matter.  

In the following diagram the negative link $f$ is contracted.
\begin{center}
\includegraphics[scale=.8]{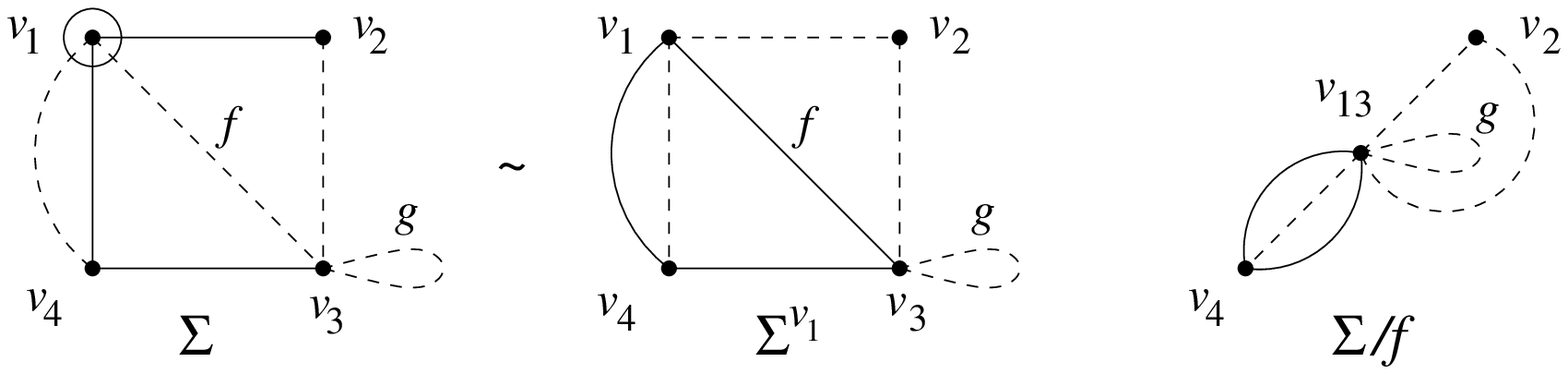}
\end{center}

\begin{lem}\label{L:linkcontraction}
In a signed graph $\Sigma$ any two contractions of a link $e$ are switching equivalent. The contraction of a link in a switching class is a well defined switching class.
\end{lem}

To contract a positive loop or a loose edge $e$, just delete $e$.

If $e$ is a negative loop or half edge and $v$ is the vertex of $e$, delete $v$ and $e$, but not any other edges.  Any other edges at $v$ lose their endpoint $v$.  A loop or half edge at $v$ becomes a loose edge.  A link with endpoints $v,w$ becomes a half edge at $w$.

In the following diagram the negative loop $g$ is contracted.
\smallskip
\begin{center}
\includegraphics[scale=.8]{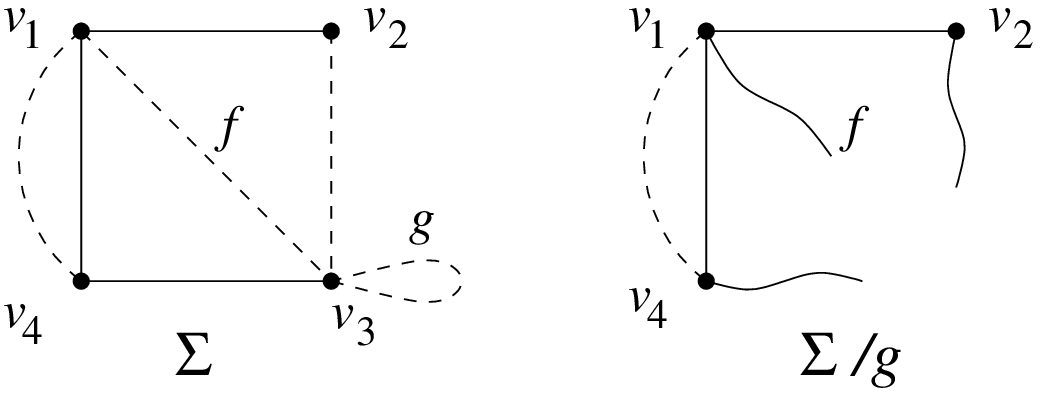}
\end{center}

%%%%%%
\subsubsection{Contracting an edge set $S$}\label{sg.contractset}\

The edge set and vertex set of $\Sigma/S$ are
$$
E(\Sigma /S) := E \setm S, \quad V(\Sigma /S) := \pib(\Sigma|S) = \pib(S).
$$
This means we identify all the vertices of each balanced component so they become a single vertex.  For $f \in E(\Sigma/S)$, the endpoints are given by the rule 
$$
V_{\Sigma/S}(f) = \{ W \in \pib(S) : w\in V_\Sigma(f) \text{ and } w \in W \in \pib(S) \}.
$$
(For instance, suppose $f$ is a loop at $w$ in $\Sigma$, so that $V_\Sigma(f) = \{w,w\}$.  If $w \in W \in \pib(S)$, then $W$ is a repeated vertex in $V_{\Sigma/S}(f)$ so $f$ is a loop in $\Sigma/S$.  If $w \in V_0(S)$, then $V_{\Sigma/S}(f) = \eset$ so $f$ is a loose edge in $\Sigma/S$.)  
To define the signature of $\Sigma/S$, first switch $\Sigma$ to $\Sigma^\zeta$ so every balanced component of $S$ is all positive.  Then 
$$\sigma_{\Sigma/S}(e) := \sigma^\zeta(e).$$

\begin{lem}\label{L:contractionequiv}
\begin{enumerate}[\rm(a)]
\item Given $S \subseteq E(\Sigma)$, all contractions $\Sigma /S$ (by different choices of how to switch $\Sigma$) are switching equivalent.  Any switching of one contraction $\Sigma/S$ is another contraction and any contraction $\Sigma^{\zeta} /S$ of a switching of\/ $\Sigma$ is a contraction of $\Sigma$.
\label{L:contractionsw}
\item If\/ $|\Sigma_1| = |\Sigma_2|$, $S \subseteq E$ is balanced in both $\Sigma_1$ and $\Sigma_2$, and $\Sigma_1 /S$ and $\Sigma_2 /S$ are switching equivalent, then $\Sigma_1$ and $\Sigma_2$ are switching equivalent.
\label{L:contractionswequiv}
\item For $e \in E$, $[\Sigma/e]$ and\/ $[\Sigma/\{e\}]$ are essentially the same switching class.
\label{L:contractionswclass}
\end{enumerate}
\end{lem}

Part \eqref{L:contractionsw} means that the switching class $[\Sigma/S]$ is uniquely defined, even though the signed graph $\Sigma/S$ is not unique.  Part \eqref{L:contractionswclass} means that $[\Sigma/e] = [\Sigma/\{e\}]$ except for details of notation.

%%%%%%
\subsubsection{Minors}\label{sg.minor}\

A \emph{minor} of $\Sigma$ is any contraction of any subgraph.

\begin{thm}\label{T:minors}
A minor of a minor is a minor.  Thus, the result of any sequence of deletions and contractions of edge and vertex sets of\/ $\Sigma$ is a minor of\/ $\Sigma$.  
\end{thm}

\begin{proof}
The proof is technical but not deep.  See Zaslavsky (1982a), Proposition 4.2.
\end{proof}

%%%%%%%%%%%%
\subsection{Frame circuits}\label{circuit}\

A \emph{frame circuit} of $\Sigma$ is a subgraph, or edge set, that is either a positive circle or a loose edge, or a pair of negative circles that intersect in precisely one vertex and no edges (this is a \emph{tight handcuff circuit}), or a pair of disjoint negative circles together with a minimal path that connects them (this is a \emph{loose handcuff circuit}).  We regard a tight handcuff circuit as having a connecting path of length 0 (it is the common vertex of two the circles).  A half edge and a negative loop are equivalent in everything that concerns frame circuits; a `negative circle' in the definition may be a half edge.

The three kinds of frame circuit:

\bigskip
\hfill
{\includegraphics[scale=.6]{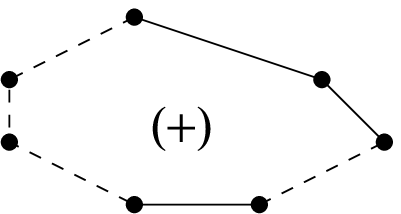}}
\hfill
{\includegraphics[scale=.6]{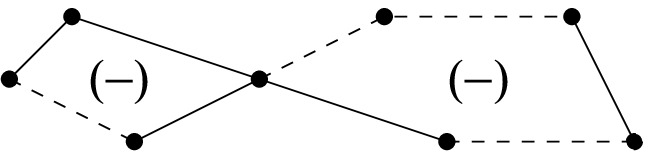}}
\hfill
{\includegraphics[scale=.6]{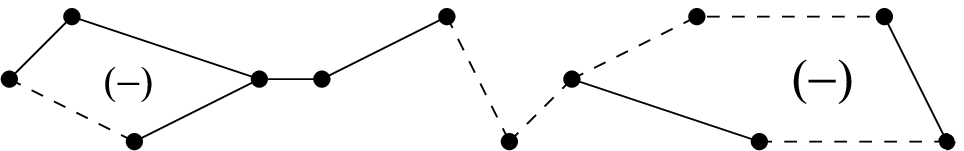}}
\hfill
\bigskip

In $+\Gamma$ (if $\Gamma$ has no half edges), a frame circuit is simply a circle or a loose edge.

\begin{prop}\label{P:handcuffs}
$\Sigma$ contains a loose handcuff circuit if and only if there is a component of\/ $\Sigma$ that contains two vertex-disjoint negative circles.
\end{prop}

The proof is elementary.  The next is less elementary.

\begin{prop}\label{P:ehandcuff}
Let $e$ be an edge in $\Sigma$.  
Then $e$ is contained in a frame circuit if and only if $e$ is not a partial balancing edge.
\end{prop}

\begin{proof}
If $e$ lies in a balanced component of $\Sigma$, it is in a frame circuit $\iff$ it is in a circle $\iff$ it is not an isthmus $\iff$ it is not a partial balancing edge.  Therefore we may assume $e$ is in an unbalanced component $\Sigma'$.
\smallskip

\emph{Necessity.}  
If $e$ is in a frame circuit $C$, then $\Sigma'$ contains $C$.  
\smallskip

\noindent\parbox{3.5in}{\quad
If $e$ is an isthmus of $C$, then $\Sigma' \setm e$ contains both negative circles of $C$.  If $\Sigma' \setm e$ is disconnected, each of its two components contains one of those negative circles.  Therefore, $e$ is not a partial balancing edge.  
}
\hfill\raisebox{-.8cm}{\includegraphics[scale=.6]{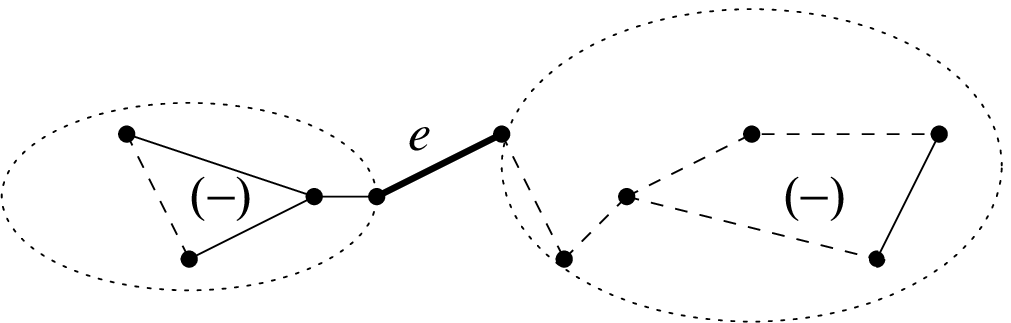}}\hfill	% figure.
\medskip

\noindent\parbox{3.5in}{\quad
If $e$ belongs to a circle in $C$, then $\Sigma' \setm e$ is connected.  Suppose $C$ is unbalanced; then $C \setm e$ is unbalanced so $\Sigma' \setm e$ is unbalanced; thus, $e$ is not a partial balancing edge.  
}
\hfill\raisebox{-1.1cm}{\includegraphics[scale=.6]{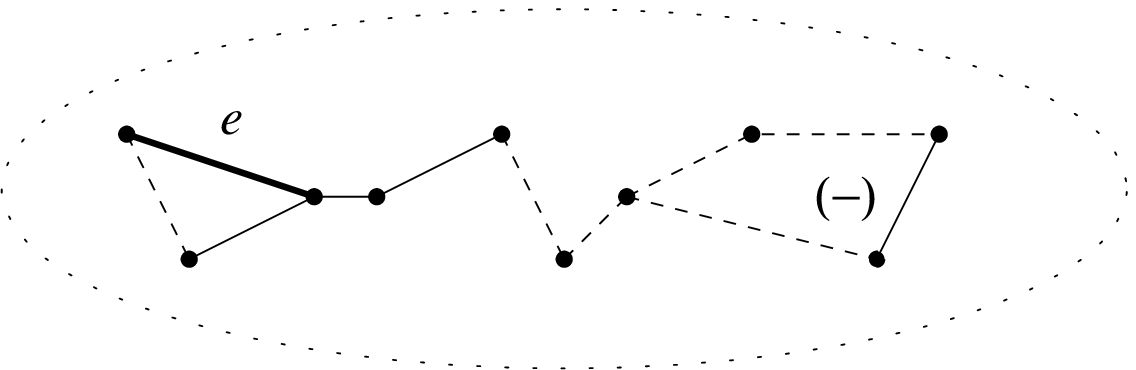}}\hfill	% figure.
\medskip

\noindent\parbox{3.5in}{\quad
Suppose to the contrary that $C$ is a positive circle.  As there is a negative circle $D$ in $\Sigma'$, for $e$ to be a partial balancing edge it must belong to $D$; we show this leads to a contradiction.  If $C \cup D \setm e$ were balanced, it could be switched to be all positive and then, as $D$ is negative, $e$ would be negative in the switched graph, but that would contradict the positivity of $C$.  Thus, $C \cup D \setm e$ is unbalanced; therefore it contains a negative circle, so $\Sigma' \setm e$ is unbalanced and $e$ is not a partial balancing edge.
}
\hfill\raisebox{-.9cm}{\includegraphics[scale=.6]{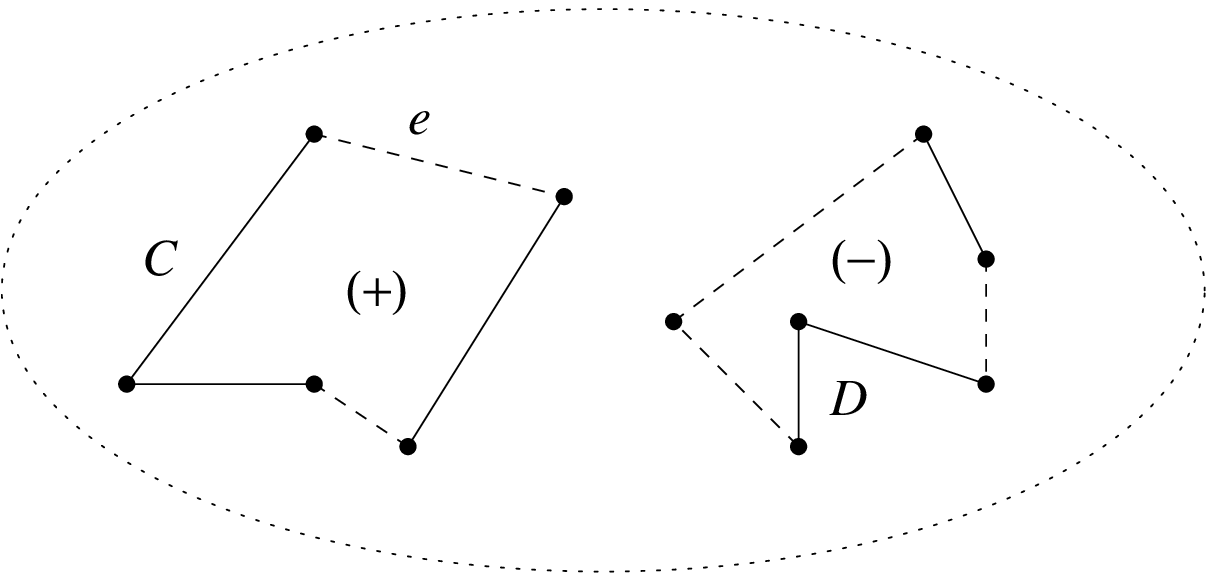}}\hfill	% figure.
\smallskip

\emph{Sufficiency.}  
Conversely, suppose $e$ is not a partial balancing edge; we produce a frame circuit $C$ containing $e$.  As $\Sigma' \setm e$ is unbalanced, it has a negative circle $D$.  
If $e$ is an unbalanced edge (a half edge or negative loop) at $v$, there is a path $P$ in $\Sigma'$ from $v$ to $D$; then $C = D \cup P \cup e$.  

If $e$ is a balanced edge, it is a link with endpoints $v,w$.  

If it is an isthmus, then $\Sigma' \setm e$ has two components, both unbalanced (by Proposition \ref{P:baledge}), so $C$ is a negative circle in each of those components together with a connecting path (which must contain $e$).  

If $e$ is not an isthmus, it lies in a circle $C'$.  If $C'$ is positive, let $C = C'$.  But suppose $C'$ is negative; then there are three subcases, depending on how many points of intersection $C'$ has with $D$.  If there are no such points, take a minimal path $P$ connecting $C'$ to $D$ and let $C = D \cup P \cup C'$.  If there is just one such point, $C = D \cup C'$.  If there are two or more such points, take $P$ to be a maximal path in $C'$ that contains $e$ and is internally disjoint from $D$.  Then $P \cup D$ is a theta graph in which $D$ is negative; hence one of the two circles containing $P$ is positive, and this is the circuit $C$.
\end{proof}

Proposition \ref{P:handcuffs} suggests vertex-disjoint negative circles are important, which is true.  There is an important theorem about when they do not exist.

\begin{thm}[Slilaty (2007a)]\label{T:2disjointnegcircles}
$\Sigma$ has no two vertex-disjoint negative circles if and only if one or more of the following is true:
\begin{enumerate}[{\rm(a)}]
\item  $\Sigma$ is balanced,
\item  $\Sigma$ has a balancing vertex,
\item  $\Sigma$ embeds in the projective plane, or
\item $\Sigma$ is one of a few exceptional cases.
\end{enumerate}
\end{thm}

We will not discuss embedding in the projective plane, which is a large topic in itself; see Zaslavsky (1993a)  Archdeacon and Debowsky (2005a).

%%%%%%%%%%%%
\subsection{Closure and closed sets}\label{clos}\

Closure of edge sets is a main property of ordinary graphs.  It generalizes to signed graphs but it becomes much more complicated.

%%%%%
\subsubsection{Closure in signed graphs}\label{sgclos}\

The definition of closure begins with an auxiliary operation.  The \emph{balance-closure} of an edge set $S$ is
$$
\bcl(S) := S \cup \{ e \in S^c: \exists \text{ a positive circle } C \subseteq S \cup e \text{ such that } e \in C \} \cup E_0(\Sigma) .
$$
The \emph{closure} of $S$ is 
$$
\clos(S) := \big(E{:}V_0(S) \big) \cup \Big( \bigcup_{i=1}^k \bcl(S_i) \Big) \cup E_0(\Sigma),
$$
where $S_1, \ldots, S_k$ are the balanced components of $S$.

An edge set is \emph{closed} if it equals its own closure: $\clos S = S$.  We write
$$
\Lat\Sigma := \{ S \subseteq E : S \text{ is closed} \}.
$$
When partially ordered by set inclusion, $\Lat\Sigma$ is a lattice.

A half edge and a negative loop are equivalent in everything that concerns closure.

The usual closure operator in a graph $\Gamma$ is the same as closure in $+\Gamma$.  By way of comparison, observe that the only kind of frame circuit in an ordinary graph is a circle---any circle, since in $+\Gamma$ every circle is positive.  That makes graph closure simple.

%%%%%
\subsubsection{Balance properties}\label{closbal}\

\begin{lem}\label{L:bcl}
For an edge set $S$, $\bcl(S)$ is balanced if and only if $S$ is balanced.  Furthermore, if $S$ is balanced, $\bcl(\bcl S) = \bcl(S) = \clos(S)$.
\end{lem}

\begin{lem}\label{L:balptn}
For an edge set $S$, $\pib(\clos S) = \pib(\bcl S) = \pib(S)$ and\/ $V_0(\clos S) = V_0(\bcl S) = V_0(S)$.
\end{lem}

%%%%%
\subsubsection{Abstract closure}\label{absclos}\

Let $E$ be any set; its \emph{power set} is the class $\cP(E)$ of all subsets of $E$.  A function $J: \cP(E) \to \cP(E)$ is an \emph{abstract closure operator on $E$} if it has the three properties
\begin{enumerate}
\item[(C1)] $J(S) \supseteq S$ for every $S \subseteq E$ (increase).
\item[(C2)] $R \subseteq S \implies J(R) \subseteq J(S)$ (isotonicity).
\item[(C3)] $J(J(S)) = J(S)$ (idempotence).
\end{enumerate}

\begin{thm}\label{T:closure}
The operator $\clos$ on subsets of\/ $E(\Sigma)$ is an abstract closure operator.
\end{thm}

\begin{proof}
The definition makes it clear that $\clos$ is increasing and isotonic.  What remains to be proved is that $\clos(\clos(S)) = \clos(S)$.  

Let $\pib(S) = \{B_1, \dots, B_k\}$; thus, $S{:}B_i$ is balanced.  By the definition of closure and Lemma \ref{L:balptn}, 
\begin{align*}%\label{E:}
\clos(\clos S) &= \big(E{:}V_0(\clos S)\big) \cup \bigcup_{i=1}^k \bcl\big((\clos S){:}B_i\big) \cup E_0(\Sigma) \\
&= \big(E{:}V_0(S)\big) \cup \bigcup_{i=1}^k \bcl\big((\bcl S){:}B_i\big) \cup E_0(\Sigma) \\
&= \big(E{:}V_0(S)\big) \cup \bigcup_{i=1}^k \bcl(S{:}B_i) \cup E_0(\Sigma) 
= \clos S.
\qedhere
\end{align*}
\end{proof}

%%%%%
\subsubsection{Matroid closure}\label{matroidclos}\

The closure operator of a signed graph has an additional property, the exchange property, whose theory is the theory of matroids.  That is, $\clos_\Sigma$ is a matroid closure.  Matroids are too complicated to explain here; see \cite{Oxley}.  One aspect of matroid closure we do want is:

\begin{thm}\label{T:cctclosure}
For $S \subseteq E$, 
$$
\clos S = S \cup \{ e \notin S : \exists \text{ a frame circuit $C$ such that } e \in C \subseteq S \cup e \}.
$$
\end{thm}

\begin{proof}
Both parts of the proof depend heavily on Proposition \ref{P:ehandcuff}.  
We treat a half edge as if it were a negative loop, and for simplicity we neglect loose edges.
\medskip

\noindent\emph{Necessity.}
We want to prove that if $e \in \clos S$, a frame circuit $C$ exists.  Let $S'$ be the component of $S \cup \{e\}$ that contains $e$.

If $S'$ is contained in one of the sets $\bcl S_i$, then $C$ exists by the definition of balance-closure.

Assume $S' \subseteq E{:}V_0(S)$.  Then $S' \setm e$ consists of one or two components of $S{:}V_0(S)$.  Every such component is unbalanced, so $S'$ is unbalanced and $e$ is not a partial balancing edge of it.  By Proposition \ref{P:ehandcuff}, $e$ is contained in a frame circuit $C \subseteq S \cup \{e\}$.
\medskip

\noindent\emph{Sufficiency.}
Assuming a circuit $C$ exists, we want to prove that $e \in \clos S$.  

If $C$ is balanced, $e \in \bcl S \subseteq \clos S$.

If $C$ is unbalanced, the component $S'$ of $S \cup \{e\}$ that contains $e$ is unbalanced.  By Proposition \ref{P:ehandcuff}, $e$ is not a partial balancing edge of $S \cup \{e\}$; therefore $S' \setm e$ has only unbalanced components.  It follows that $V(S') \subseteq V_0(S)$, so $e \in C \subseteq E{:}V_0(S) \subseteq \clos S$.
\end{proof}

A consequence is that $\Lat\Sigma$ is a geometric lattice; but that is too much matroid theory for here.

%%=================================%%
\subsection{Oriented signed graphs = bidirected graphs}\label{bdg}\

Bidirected graphs were introduced by Edmonds and first published in a paper on matching theory, Edmonds and Johnson (1970a).  Later, Zaslavsky (1991b) found that they are oriented signed graphs.

A \emph{bidirected graph} is a graph in which each end of each edge has an independent direction.  Thus, an oriented signed graph is a bidirected graph.  
Formally, a bidirected graph $\Beta$ (read `Beta') is a pair $(\Gamma,\tau)$ where $\Gamma$ is a graph and $\tau$ is a function from edge ends to $\{+,-\}$.  If $e$ has endpoints $v,w$ and we write $(v,e)$ for the end of edge $e$ at vertex $v$, then $\tau(v,e) = +$ if the end is directed towards $v$ and $= -$ if the end is directed away from $v$.  (The two directions on $e$ agree when $\tau(v,e) = -\tau(w,e)$, which may at first sight seem peculiar.)

An orientation of an ordinary graph gives a direction to each edge.  An \emph{orientation of a signed graph} gives a direction to each end of each edge.  If $e$ is positive, the directions at the two ends of $e$ must agree in pointing from one endpoint to the other.  If $e$ is negative, the directions at the two ends of $e$ must disagree; that is, they both point towards the middle of $e$ (an \emph{introverted} edge) or both away from the middle (an \emph{extraverted} edge).

A bidirected graph has an edge signature: 
$$\sigma_\Beta(e) = -\tau(v,e)\tau(w,e).$$
That is, if the directions at the two ends agree, the edge is positive; if they disagree, the edge is negative.  Thus, bidirected graphs and oriented signed graphs are exactly the same thing; only the name is different.  

Here is an example in which all negative edges are extraverted: 
\hfill\raisebox{-1.5cm}{\includegraphics[scale=.7]{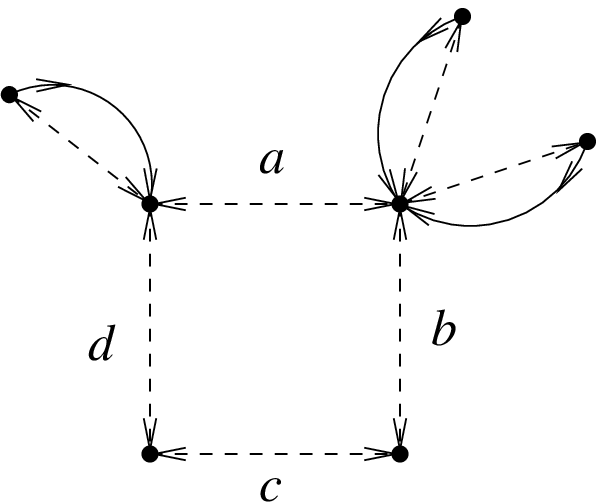}}
\medskip

We write $|\Beta|$ for the underlying graph of $\Beta$ and $\Sigma_\Beta$ for the signed graph $(|\Beta|,\sigma_\Beta)$.
$\Beta$ is switched by the rule $\Beta^\zeta := (|\Beta|,\tau^\zeta)$ where 
$$\tau^\zeta(v,e) := \tau(v,e)\zeta(v).$$

\begin{lem}\label{L:swbd}
$\Sigma_\Beta$ and $\Sigma_{\Beta^\zeta}$ are switching equivalent; in fact, $\Sigma_{\Beta^\zeta} = (\Sigma_\Beta)^\zeta$.
\end{lem}

\begin{proof}
Let $e$ have endpoints $v,w$.  Then 
\begin{align*}
\sigma_\Beta^\zeta(e) &= \zeta(v)\sigma(e)\zeta(w) = \zeta(v)[-\tau(v,e)\tau(w,e)]\zeta(w) \\
&= -[\tau(v,e)\zeta(v)][\tau(w,e)\zeta(w)] = \sigma_{\Beta^\zeta}(e).
\qedhere
\end{align*}
\end{proof}

In contrast to switching of a signed graph, $\Beta^\zeta \neq \Beta^{-\zeta}$; indeed, $\tau^{-\zeta} = -\tau^\zeta$.

\subsubsection*{Acyclic orientations}\

In an orientation $\tau$ of a signed graph, a vertex $v$ is a \emph{source} if $\tau(v,e) = +$ for every edge end $(v,e)$ at $v$; in other words, all edges incident with $v$ point into it.  It is a \emph{sink} if $\tau(v,e) = -$ for every edge end at $v$; i.e., all edges at $v$ point away from it.  An orientation $\tau$ of $\Sigma$ is called \emph{acyclic} if for every frame circuit $C\subseteq\Sigma$, the oriented subgraph $(\Sigma|C,\tau|_C)$ has a source or a sink, where $\tau|_C$ denotes the restriction to the edge ends $(v,e)$ in $\Sigma|C$.

\subsubsection*{Digraphs}\

The two arrows that point in a consistent direction along a positive edge can be just as well represented by a single arrow.  Thus, an oriented all-positive signed graph---or we may call it an all-positive bidirected graph---is a directed graph; only the name has changed.  It follows that digraphs fall under the topic of bidirected graphs.  I have an impression that much of digraph theory does not readily generalize to bidirected graphs, but it may be that it has merely not yet been tried.

%%===================================================%%
\section{Geometry and Matrices}\label{geom}

In this section we write the vertex set as $V = \{v_1,v_2,\ldots,v_n\}$.  $\F$ denotes any field.  The most important field here will be $\bbR$, the real number field.  

%%%%%%%%%%%%
\subsection{Vectors for edges}\label{vectors}\

We have a signed graph $\Sigma$ of order $n$.  For each edge $e$ there is a vector $\bx(e) \in \F^n$, whose definition is, for the four types of edge: 
$$
\begin{matrix}
%Begin matrix row.
%Link matrix
\begin{matrix}
  \begin{array}{rl}
  \begin{matrix} \\ \\ \\[2pt] i  \\ \\ \\ \\[6pt] j  \\  \\ \\ \\ \end{matrix} 
  & 
  \begin{bmatrix} 0 \\ \vdots \\ 0 \\ \pm 1 \\ 0 \\ \vdots \\ 0 \\ \mp \sigma(e) \\ 0  \\ \vdots \\ 0 \\ \end{bmatrix}
  \end{array} \\ 
\end{matrix}
%End link matrix
&& 
%+ link matrix
\begin{matrix}
  \begin{array}{rl}
  \begin{bmatrix} 0 \\ \vdots \\ 0 \\ \pm 1 \\ 0 \\ \vdots \\ 0 \\ \mp1 \\ 0  \\ \vdots \\ 0 \\ \end{bmatrix}
  \end{array} \\ 
\end{matrix}
%End + link matrix
&& 
%- link matrix
\begin{matrix}
  \begin{array}{rl}
  \begin{bmatrix} 0 \\ \vdots \\ 0 \\ \pm 1 \\ 0 \\ \vdots \\ 0 \\ \pm1 \\ 0  \\ \vdots \\ 0 \\ \end{bmatrix}
  \end{array} \\ 
\end{matrix}
%End - link matrix
&
& 
%Loop matrix
\begin{matrix}
  \begin{matrix}%{rl}
   &   \\ \\ 
   \begin{matrix} \\ \\ \\ i  \\ \\ \\ \\  \end{matrix} 
  & 
  \begin{bmatrix}
  0 \\ \vdots \\ 0 \\ \pm (1 - \sigma(e)) \\ 0 \\ \vdots \\ 0 \\
  \end{bmatrix}     \\\\\\\\
  \end{matrix} \\ 
\end{matrix}
%End loop matrix
&
& 
%Half-edge matrix
\begin{matrix}
  \begin{matrix}%{rl}
   &   \\ 
   \begin{matrix} \\ \\ \\ i  \\ \\ \\ \\ \end{matrix} 
  & 
  \begin{bmatrix} 0 \\ \vdots \\ 0 \\ \pm 1 \\ 0 \\ \vdots \\ 0 \\ \end{bmatrix} \\ \\ \\
  \end{matrix} \\ 
\end{matrix}
%End half-edge matrix
&
%& 
%Loose-edge matrix
\begin{matrix} \\
  \begin{bmatrix} 0 \\ \vdots \\ 0 \\ \vdots \\ 0 \\ \end{bmatrix} \\ \\ \\ 
\end{matrix}
%End loose-edge matrix
%End of matrix row.
\\
%Begin label row.
\quad\text{link }e{:}v_iv_j,
&&
+\text{ link, }
&&
-\text{ link, }
&&
\quad \text{loop }e{:}v_iv_i,
&&
\ \ \text{half edge }e{:}v_i,
&
\ \text{loose edge.} 
%End of label row.
\end{matrix}
$$
Note that the vector of a positive loop is $\bf0$.  
These vectors are well defined only up to sign, i.e., the negative of $\bx(e)$ is another possible choice of $\bx(e)$.  We make an arbitrary choice $\bx(e)$ for each edge $e$, which does not affect the linear dependence properties.  

For a set $S \subseteq E$, define $\bx(S) := \{ \bx(e) : e \in S \}$.

\begin{thm}\label{T:dependence} 
Let $S$ be an edge set in\/ $\Sigma$ and consider the corresponding vector set\/ $\bx(S)$ in the vector space\/ $\F^n$ over a field\/ $\F$.
\begin{enumerate}[{\rm(a)}]
\item When $\Char\F \neq 2$, $\bx(S)$ is linearly dependent if and only if $S$ contains a frame circuit.
\item When $\Char\F = 2$, $\bx(S)$ is linearly dependent if and only if $S$ contains a circle or a loose edge. 
\end{enumerate}
\end{thm}

The theorem is a restatement of Zaslavsky (1982a), Theorem 8B.1. 
The proof, which we omit, is neither very short nor very long.  

\begin{cor}\label{C:mindep}
If $\Char\F \neq 2$, the minimal linearly dependent subsets of\/ $\bx(E)$ are the sets $\bx(C)$ where $C$ is a frame circuit in $\Sigma$.
\end{cor}

Define $S \subseteq E(\Sigma)$ to be \emph{independent} if the vectors in $\bx(S)$ are linearly independent (and distinct from each other) over a field whose characteristic is not $2$.  

\begin{cor}\label{C:sgindep}
A set $S \subseteq E(\Sigma)$ is independent if and only if it does not contain a frame circuit.
\end{cor}

Independent sets are the signed-graphic generalization of forests.  Corollary \ref{C:sgindep} implies that a set in a balanced signed graph is independent if and only if it is a forest.  Zaslavsky (1982a), Theorem 5.1(c) has a direct description of independent sets in general.

The vector subspace generated by a set $X \subseteq \F^n$ is denoted by $\langle X \rangle$.  We write
$$
\cL_\F(\Sigma) := \{ \langle X \rangle : X \subseteq \bx(E) \}.
$$
When partially ordered by set inclusion, $\cL_\F(\Sigma)$ is a lattice. 

\begin{cor}\label{C:span}
Assume that $\Char\F \neq 2$.  For $S \subseteq E(\Sigma)$, $\bx(E) \cap \langle \bx(S) \rangle = \bx(\clos S)$.  Thus, 
\begin{enumerate}[{\rm(a)}]
\item $\cL_\F(\Sigma) \cong \Lat \Sigma$, and  
\item $S$ is independent if and only if it is not in the closure of any proper subset of itself.  
\end{enumerate}
\end{cor}

The \emph{rank} of $S \subseteq E$ is defined to be
$$
\rk S := n - b(S).
$$
The rank of $\Sigma$ is $\rk\Sigma := \rk E = n - b(\Sigma)$.

\begin{thm}\label{T:rank}
Assume that $\Char\F \neq 2$.  Then $\dim \langle \bx(S) \rangle = \rk S$ for $S \subseteq E(\Sigma)$.
\end{thm}

\begin{proof}
The proof is simplest when expressed in terms of the frame matroid (Section \ref{matroid}), so I omit it; see Zaslavsky (1982a), Theorem 8B.1 and following remarks.  The essence of the proof is using Corollary \ref{C:sgindep} to compare the minimum number of edges required to generate $S$ by closure in $\Sigma$ to the minimum number of vectors $\bx(e)$ required to generate $\langle \bx(S) \rangle$.
\end{proof}

\subsubsection*{Why the signs?}\

It is natural to wonder, at first sight, why the signs in a link vector come out as they do, opposite for a positive edge and similar for a negative edge rather than the reverse.  There are several reasons.  The strongest is the correlation between the linear dependencies of edge vectors and frame circuits in $\Sigma$ as stated in Theorem \ref{T:dependence}.  If we were to adopt the opposite correspondence of edge sign with vector signs so that a positive edge had two entries with the same sign, we would be forced to define frame circuits in terms of circle signs in $-\Sigma$ instead of $\Sigma$.  That would be unnecessarily complicated.

Another justification for our sign convention is in the equations of edge hyperplanes; see Section \ref{arr}.

\subsubsection*{Orientation}\

Choosing $\bx(e)$ or $-\bx(e)$ corresponds to choosing an orientation of $\Sigma$.  Orient $\Sigma$ as $\Beta = (|\Sigma|,\tau)$, and define 
\begin{equation}
\eta_{ve} := \sum_{\text{incidences $(v,e)$}} \tau(v,e).
\label{E:taueta}
\end{equation}
Then $\bx(e)_v = \eta_{ve}$.  Conversely, if we choose $\bx(e)$ first, there is a unique $\tau$ that orients $\Sigma$ and satisfies \eqref{E:taueta}, with the exception that the orientation of a positive loop is arbitrary.

%%%%%%%%%%%%
\subsection{The incidence matrix}\label{i}\

The \emph{incidence matrix} $\Eta(\Sigma) = \begin{pmatrix} \eta_{ve} \end{pmatrix}_{v\in V, e\in E}$ (read `Eta of Sigma') is a $V \times E$ matrix (thus, it has $n$ rows and $m$ columns where $m := |E|$) in which the column corresponding to edge $e$ is the column vector $\bx(e)$; that is, 
$$\Eta(\Sigma) := \begin{bmatrix} \bx(e_1) & \bx(e_2) & \cdots & \bx(e_m) \end{bmatrix}.$$
This matrix of a signed graph has a role in more than geometry; it underlies both the Laplacian matrix and the line graph of $\Sigma$.

A small example:
\par
\begin{center}
\includegraphics[scale=.7]{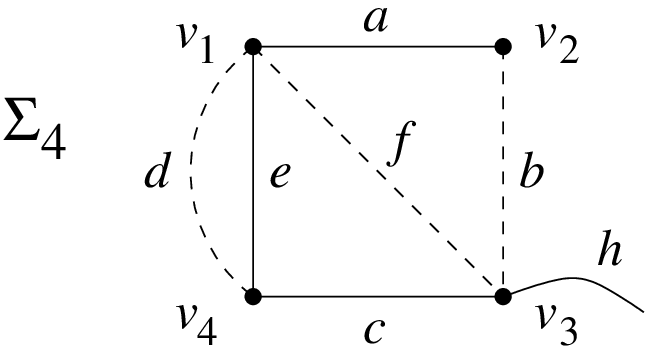}	% Small s.g. for incidence matrix.
\qquad
\raisebox{1.3 cm}{
$\Eta(\Sigma_4) = 
\begin{array}{c}
\begin{array}{ccccccc}
\ a\ & b\ \,& c\ \,& d\ \,& e\ &\, f\ & h 
\end{array}
\\[6pt]
\begin{pmatrix}
1	&0	&0	&1	&-1	&-1	&\,0	\\
-1	&1	&0	&0	&0	&0	&0	\\
0	&1	&1	&0	&0	&-1	&1	\\
0	&0	&-1	&1	&1	&0	&0	\\
\end{pmatrix}
\end{array}$
}
\end{center}

\begin{thm}\label{T:irank}
Over a field whose characteristic is not\/ $2$, the rank of\/ $\Eta(\Sigma)$ is $\rk\Sigma = n - b(\Sigma)$ and, for $S \subseteq E$, the rank of\/ $\Eta(\Sigma|S)$ is $\rk S = n - b(S)$.
\end{thm}

\begin{proof}
The column rank is the dimension of the span of the columns corresponding to $S$, which is the span of $\bx(S)$.  Apply Theorem \ref{T:rank}.
\end{proof}

%%%%%%%%%%%%
\subsection{Frame matroid}\label{matroid}\

The \emph{frame matroid} $G(\Sigma)$ is an abstract way of describing all the previous vector-like characteristics of a signed graph:  dependent edge sets, minimal dependencies, rank, closure, and closed sets.  See Zaslavsky (1982a), Section 5, for details.  For these aspects of matroid theory consult the early chapters of the excellent text by Oxley \cite{Oxley}.

I mention matroids here because in the frame matroid $G(\Sigma)$ we have a notion of \emph{independent edge set} which expresses abstractly, on an equal footing, both the independence of representation vectors and the corresponding independence a set of edges whose columns in $\Eta(\Sigma)$ are linearly independent.  Hence, theorems of matroid theory can be useful in signed-graphic geometry; but that is a road we do not take here.

%%%%%%%%%%%%
\subsection{The adjacency and Laplacian (Kirchhoff) matrices}\label{ak}\

The \emph{adjacency matrix} is $A(\Sigma) = (a_{ij})_{n\times n}$ defined by 
$$a_{ii} := 0 + h_{v_i} + 2l_{v_i}^+ - 2l_{v_i}^-,$$
where $h_v$ is the number of half edges at $v$ and $l_v^\epsilon$ is the number of loops at $v$ with sign $\epsilon$, and for $i \neq j$, 
$$a_{ij} := \text{ (the number of positive edges $v_iv_j$) $-$ (the number of negative edges $v_iv_j$)}.$$  
An important fact about the adjacency matrix is that it does not change if a parallel pair of edges, one positive and one negative, is deleted from $\Sigma$ (this is \emph{cancellation of a negative digon}), or if a loose edge or positive loop is deleted.  A signed link graph is \emph{reduced} if it has no such parallel pairs, and no positive loops or loose edges.  Up to isomorphism there is a unique reduced signed graph $\bar\Sigma$ with the same adjacency matrix as $\Sigma$.

The \emph{Laplacian matrix} or \emph{Kirchhoff matrix} of $\Sigma$ is $L(\Sigma) := D(|\Sigma|) - A(\Sigma)$, where $D(|\Sigma|)$, called the \emph{degree matrix}, is the diagonal matrix whose diagonal element $d_{ii} = d_{|\Sigma|}(v_i)$.  
(Recall that a loop counts twice in the degree of an unsigned graph.)  
I mention the Laplacian here because it gives us a second application of the incidence matrix.

We think of these as matrices over the complex numbers so we can talk about their eigenvalues and eigenvectors.

Some examples:  
\begin{itemize}
\item $A(-\Sigma) = -A(\Sigma)$.
\item $A(+\Gamma) = A(\Gamma)$, the adjacency matrix of $\Gamma$.
\item $A(-\Gamma) = -A(\Gamma)$.
\item $L(+\Gamma) = L(\Gamma) := D(\Gamma) - A(\Gamma)$, the Laplacian (or Kirchhoff) matrix of $\Gamma$.
\item $L(-\Gamma) = Q(\Gamma) := D(\Gamma) + A(\Gamma)$, the so-called `signless Laplacian matrix' (or '$Q$-matrix') of $\Gamma$, which has recently been studied intensively; in signed graph theory it is simply another Laplacian matrix, though an especially important one.
\end{itemize}

A particular example is $\Sigma_4$ from Section \ref{i}:
$$A(\Sigma_4) = 
\begin{pmatrix}
0	&1	&-1	&\ 0	\\
1	&0	&-1	&\ 0	\\
-1	&-1	&1	&\ 1	\\
0	&0	&1	&\ 0	\\
\end{pmatrix},
\qquad
L(\Sigma_4) = 
\begin{pmatrix}
4	&-1	&1	&\ 0	\\
-1	&2	&1	&\ 0	\\
1	&1	&3	&\ -1	\\
0	&0	&-1	&\ 3	\\
\end{pmatrix}$$

\begin{prop}\label{P:kirchhoff}
For a signed graph $\Sigma$, $L(\Sigma) = \Eta(\Sigma) \Eta(\Sigma)\transpose.$
\end{prop}

The proof is by matrix multiplication---the same straightforward calculation as with ordinary graphs.  

The eigenvalues of $A(\Sigma)$ are known as the \emph{eigenvalues of\/ $\Sigma$}.  Those of $L(\Sigma)$ are called the \emph{Laplacian eigenvalues} of $\Sigma$.  
Every graph in the switching class $[\Sigma]$ has the same spectrum.

\begin{thm}\label{T:e}
The eigenvalues of $A(\Sigma)$ are real.  

The eigenvalues of $L(\Sigma)$ are real and non-negative.

The eigenvalues of each matrix are unchanged by switching $\Sigma$.
\end{thm}

\begin{proof}
As with ordinary graphs, $A(\Sigma)$ is symmetric and $L(\Sigma) = \Eta(\Sigma) \Eta(\Sigma)\transpose$ is positive semidefinite.  

Switching $\Sigma$ by $\zeta$ has the effect on the adjacency and Laplacian matrices of conjugating them by a diagonal matrix $D(\zeta)$ whose diagonal entries are the values $\zeta(v_i)$; i.e., 
$$A(\Sigma^\zeta) = D(\zeta)A(\Sigma)D(\zeta) \quad \text{ and } \quad L(\Sigma^\zeta) = D(\zeta)L(\Sigma)D(\zeta) .$$  
The conjugation does not alter the eigenvalues.
\end{proof}

\subsubsection*{A use for the Laplacian}\

\begin{thm}[Matrix-Tree Theorem for Signed Graphs]\label{T:matrixtree}
Let $b_i :=$ the number of sets of\/ $n$ independent edges in $\Sigma$ that contain exactly $i$ circles.  Then $\det L(\Sigma) = \sum_{i=0}^n 4^i b_i.$
\end{thm} 

The proof uses the Cauchy-Binet Theorem in the same way as it is used to prove the Matrix-Tree Theorem for ordinary graphs.  Note that the $i$ circles must all be negative for the edge set to be independent.  
Chaiken (1982a) has a generalization to signed digraphs and to arbitrary minors of the Laplacian matrix.

%%%%%%%%%%%%
\subsection{Arrangements of hyperplanes}\label{arr}\

An \emph{arrangement of hyperplanes} in $\bbR^n$, $\cH = \{ h_1, h_2, \ldots, h_m\}$, is a finite set of hyperplanes.  A \emph{region} of $\cH$ is a connected component of the complement, $\bbR^n \setm \big( \bigcup_{k=1}^m h_k \big)$.  We write $r(\cH) :=$ the number of regions.  The \emph{intersection lattice} is the family $\cL(\cH)$ of all subspaces that are intersections of hyperplanes in $\cH$, partially ordered by reverse inclusion, $s \leq t \iff t \subseteq s$.  The \emph{characteristic polynomial} of $\cH$ is
\begin{equation}\label{E:charpoly}
p_\cH(\lambda) := \sum_{\cS \subseteq \cH} (-1)^{|\cS|} \lambda^{\dim\cS} ,
\end{equation}
where $\dim\cS := \dim\big( \bigcap_{h_k \in \cS} h_k \big).$

\begin{thm}[{\cite[Theorem A]{FUTA}}]\label{T:futa}
$r(\cH) = (-1)^n p_\cH(-1).$
\end{thm}

\subsubsection{The signed-graphic hyperplane arrangement}\

A signed graph $\Sigma$, with edge set $\{ e_1, e_2, \ldots, e_m \}$, gives rise to a hyperplane arrangement 
$$
\cH[\Sigma] := \{ h_1, h_2, \ldots, h_m\} 
$$
where $h_k$ is the solution set of the equation $\bx(e_k) \cdot \bx = 0$; i.e., $$h_k = \{ \bx \in \bbR^n : \bx(e_k) \cdot \bx = 0 \}.$$  (By $\cdot$ I mean the dot product $\by\cdot \bx := y_1x_1+\cdots+y_nx_n$.)
In terms of the graph, 
\begin{align*}%\label{E:}
h_k &\text{ has the equation } 
\begin{cases}
x_j = \sigma(e_k) x_i, &\text{ if $e_k$ is a link or loop with endpoints } v_i, v_j, \\
x_i = 0, &\text{ if $e_k$ is a half edge or a negative loop at } v_i, \\
\ 0 = 0, &\text{ if $e_k$ is a loose edge or a positive loop}.
\end{cases}
\end{align*}

A positive link gives the equation $x_j = +x_i$, a negative one the equation $x_j = -x_i$, and a negative loop the equation $x_i = -x_i$, simplified here to $x_i = 0$.  The signs in the equations of edge hyperplanes agree with the edge signs; for that reason I am inclined to think hyperplanes are more fundamental than the vectors defined in Section \ref{i}.

(The equation $0=0$ of a loose edge or $x_i=x_i$ of a positive loop has the solution set $\bbR^n$, which is not truly a hyperplane, but I allow it under the name `degenerate hyperplane'.  
If the arrangement contains the degenerate hyperplane, it has no regions, because when the hyperplanes are removed from $\bbR^n$, nothing remains.  For a few reasons we cannot avoid the degenerate hyperplane; the chief is that it appears in the geometrical contraction-deletion process corresponding to that used to evaluate chromatic polynomials; see Theorem \ref{T:chiprops}.)

\begin{lem}\label{L:hypdim}
Let $\cS = \{h_{i_1},\ldots,h_{i_l}\}$ be the subset of\/ $\cH[\Sigma]$ that corresponds to the edge set $S = \{e_{i_1},\ldots,e_{i_l}\}$.  Then $\dim\bigcap\cS = b(S)$.
\end{lem}

\begin{proof}
Apply vector space duality to Theorem \ref{T:rank}.
\end{proof}

\begin{thm}\label{T:lattices}
$\cL(\cH[\Sigma])$, $\cL_\bbR(\Sigma)$, and $\Lat \Sigma$ are all isomorphic.  
\end{thm}

\begin{proof}
The isomorphism between $\cL(\cH[\Sigma])$ and $\cL_\bbR(\Sigma)$ is standard vector-space duality.  The isomorphism $\cL_\bbR(\Sigma) \cong \Lat \Sigma$ is in Corollary \ref{C:span}.
\end{proof}

\subsubsection{Acyclic orientations reappear}\

Now recall acyclic orientations from Section \ref{bdg}.  
A notable fact is that the regions of $\cH[\Sigma]$ are in bijective correspondence with the acyclic orientations of $\Sigma$.  For an orientation $\tau$ define 
$$
R(\tau) := \big\{ \bx \in \bbR^n : \tau(v_i,e)x_i+\tau(v_j,e)x_j > 0 \text{ for every edge } e = v_iv_j \big\}.
$$

\begin{thm}\label{T:regions}\
\begin{enumerate}[{\rm(a)}]
\item $R(\tau)$ is nonempty $\iff$ $\tau$ is acyclic.
\item Every region is an $R(\tau)$ for some acyclic $\tau$. 
\end{enumerate}
\end{thm}

Zaslavsky (1991b) gives two proofs.  They are not short.

%%===================================================%%
\section{Coloring}\label{col}

We color a signed graph from a color set
$$
\Lambda_k := \{ \pm 1,\pm 2,\ldots,\pm k\} \cup \{0 \}
$$  
or a zero-free color set
$$
\Lambda_k^*:=\Lambda_k \backslash \{0 \}=\{ \pm 1,\pm 2,\ldots,\pm k\}.
$$  
A \emph{$k$-coloration} (or \emph{$k$-coloring}) of $\Sigma$ is a function $\gamma: V \to \Lambda_k$.  A coloration is \emph{zero free} if it does not use the color $0$.  Coloring of signed graphs comes from Zaslavsky (1982b, 1982c).  

A coloration $\gamma$ is \emph{proper} if it satisfies all the properties
$$
\begin{cases}
\gamma(v_j) \neq \sigma(e) \gamma(v_i), &\text{ for a link or loop $e$ with endpoints } v_i, v_j, \\
\gamma(v_i) \neq 0,	&\text{ for a half edge $e$ with endpoint } v_i,
\end{cases}
$$
and there are no loose edges.  Note that these conditions for a proper coloration are opposite to the equations of the hyperplanes $h_k$.  Also note that if $\Sigma$ has a positive loop or loose edge, there are no proper colorations at all.

%%%%%%

%%%%%%%%%%%%
\subsection{Chromatic polynomials}\label{poly}\

A signed graph has two chromatic polynomials.  For an integer $k\geq0$, define 
$$
\chi_{\Sigma}(2k+1):= \text{ the number of proper $k$-colorations,} 
$$
and 
$$
\chi_{\Sigma}^*(2k):= \text{ the number of proper zero-free $k$-colorations.}  
$$
The polynomials are identically zero if and only if $\Sigma$ has a positive loop or a loose edge.

One may wonder why the variables are $2k+1$ and $2k$ and not $k$; there are several reasons, amongst which Theorem \ref{T:polynomial} and Lemma \ref{T:sgcharacteristic} and the geometrical analysis mentioned in Section \ref{geomcol} are especially important.

%%%%%%
\subsubsection{Basic properties}\

The chromatic polynomials have fundamental properties in common with ordinary graphs.

\begin{thm}\label{T:chiprops}
The chromatic polynomials have the properties of
\begin{enumerate}[{\rm(1)}]
\item Unitarity: $$\chi_{\emptyset}(2k+1)=1=\chi_{\emptyset}^*(2k) \text{ for all } k \geq 0.$$
\item Switching Invariance:  If\/ $\Sigma \sim \Sigma'$, then 
$$\chi_{\Sigma}(2k+1)=\chi_{\Sigma'}(2k+1) \quad\text{and}\quad \chi_{\Sigma}^*(2k)=\chi_{\Sigma'}^*(2k).$$

\item Multiplicativity: If\/ $\Sigma$ is the disjoint union of\/ $\Sigma_1$ and $\Sigma_2$, then 
$$\chi_\Sigma(2k+1)=\chi_{\Sigma_1}(2k+1)\chi_{\Sigma_2}(2k+1) \quad\text{and}\quad \chi_\Sigma^*(2k)=\chi_{\Sigma_1}^*(2k)\chi_{\Sigma_2}^*(2k).$$

\item Deletion-Contraction: If\/ $e$ is a link, a positive loop, or a loose edge, then 
$$
\chi_{\Sigma}(2k+1)=\chi_{\Sigma \setm e}(2k+1)-\chi_{\Sigma / e}(2k+1)
$$
and
$$
\chi_{\Sigma}^*(2k)=\chi_{\Sigma \setm e}^*(2k)-\chi_{\Sigma / e}^*(2k).
$$
\end{enumerate}
\end{thm}

\begin{proof}[Outline of Proof]
The hard part is the deletion-contraction property.  The proof is similar to the usual proof for ordinary graphs by counting proper colorations of $\Sigma\setm e$.

If $e$ is a link, first switch so it is positive.  
(Switching modifies a coloration by negating the color of every switched vertex.  The switched coloration is proper in the switched graph if and only if the original coloration was proper in the original graph.)
Then a proper coloration of $\Sigma\setm e$ gives unequal colors to the endpoints of $e$ and is a proper coloration of $\Sigma$, or it gives the same color to the endpoints and it corresponds to a proper coloration of $\Sigma/e$.  If $e$ is a half edge or a negative loop, there are two cases depending on whether the endpoint gets a nonzero color or is colored 0.

A complete proof is in Zaslavsky (1982b).
\end{proof}

These properties make it possible to prove explicit formulas which demonstrate that the chromatic polynomials are indeed polynomials.

\begin{thm}\label{T:polynomial}
$\chi_{\Sigma}(\lambda)$ is a polynomial function of\/ $\lambda=2k+1>0$; specifically, 
\begin{equation}\label{E:setexpansion}
\chi_{\Sigma}(\lambda) = \sum_{S \subseteq E} (-1)^{|S|} \lambda^{b(S)}.
\end{equation}
Also, $\chi_{\Sigma}^*(\lambda)$ is a polynomial function of\/ $\lambda=2k\geq0$.  Specifically,
\begin{equation}\label{E:balsetexpansion}
\chi^*_{\Sigma}(\lambda) = \sum_{S \subseteq E: \text{balanced}} (-1)^{|S|} \lambda^{b(S)}.
\end{equation}
\end{thm}

\begin{proof}
Apply Theorem \ref{T:chiprops} and induction on $|E|$ and $n$.
\end{proof}

Therefore, we can extend the range of $\lambda$ to all of $\bbR$.  In particular, we can evaluate $\chi_\Sigma(-1)$.  

%%%%%%
\subsubsection{A geometrical application of the chromatic polynomial}\

This lets us draw an important connection between geometry and coloring of a signed graph.

\begin{lem}\label{T:sgcharacteristic}
$\chi_\Sigma(\lambda) = p_{\cH[\Sigma]}(\lambda).$
\end{lem}

\begin{proof}
Compare the summation expressions, \eqref{E:setexpansion} and \eqref{E:charpoly}, for the two polynomials, and note that by Lemma \ref{L:hypdim} $b(S) = \dim\cS$ if $\cS \subseteq \cH[\Sigma]$ corresponds to the edge set $S$.
\end{proof}

\begin{thm}\label{T:sgregions}
The number of acyclic orientations of\/ $\Sigma$ and the number of regions of\/ $\cH[\Sigma]$ are both equal to $(-1)^n \chi_\Sigma(-1)$.
\end{thm}

%%%%%%
\subsubsection{Computational methods}\

To compute the chromatic polynomial it is often easiest to get several zero-free polynomials first and use

\begin{thm}[Zero-Free Expansion Identity]\label{T:expansion}
The chromatic and zero-free chromatic polynomials are related by
$$\chi_\Sigma(\lambda) = \sum_{W \subseteq V: \,\text{stable}} \chi^*_{\Sigma\setm W}(\lambda-1).$$
\end{thm}

\begin{proof}
Let $\lambda = 2k+1$.  For each proper $k$-coloration $\gamma$ there is a set $W := \{v \in V: \gamma(v) = 0\}$, which must be stable.  The restricted coloration $\gamma|_{V\setm W}$ is a zero-free proper $k$-coloration of $\Sigma \setm W$.  This construction is reversible.
\end{proof}

\begin{exam}\label{X:0free-exp}
A signed complete graph is $\Sigma = (K_n,\sigma)$.  In the zero-free expansion the stable vertex sets are $\eset$ and also $\{v\}$ for each $v \in V$.  Thus, 
$$\chi_\Sigma(\lambda) = \chi^*_\Sigma(\lambda-1) + \sum_{v \in V} \chi^*_{\Sigma\setm v}(\lambda-1).$$
\end{exam}

%%%%%%
\subsubsection{A geometry of coloring}\label{geomcol}\

Beck and Zaslavsky (2006a) explain why there are two chromatic polynomials of a signed graph when one is enough for ordinary graphs.  
They take a geometrical approach that views colorations in connection with the hyperplane arrangement $\cH[\Sigma]$.  Colorations correspond to certain half-integral points in $\bbR^n$ and proper colorations correspond to points that are not in a translation of $\bigcup \cH[\Sigma]$.  The two chromatic polynomials are then a natural consequence of Ehrhart's theory of counting lattice points in polytopes (for which see \cite{CCD}, for instance).

%%%%%%%%%%%%
\subsection{Chromatic numbers}\label{nos}\

The \emph{chromatic number} of $\Sigma$ is 
$$
\chi(\Sigma) := \min\{ k : \exists \text{ a proper $k$-coloration}\}, 
$$
and the \emph{zero-free chromatic number} is 
$$
\chi^*(\Sigma) := \min\{ k : \exists \text{ a zero-free proper $k$-coloration}\}.
$$
Thus, $\chi(\Sigma) = \min\{ k\geq0 : \chi_{\Sigma}(2k+1) \neq 0\}$ and $\chi^*(\Sigma) = \min\{ k\geq0 : \chi^*_{\Sigma}(2k) \neq 0\}$.

Almost any question about chromatic numbers of signed graphs is open.  The little I know about the graphs with a given value of a chromatic number is in Zaslavsky (1984a), where I studied complete signed graphs with largest or smallest zero-free chromatic number.  
Zaslavsky (1987b) also concerns chromatic number, disguised as a signed generalization of graph biparticity called `balanced decomposition number', as the balanced decomposition number of $\Sigma$ equals $\lceil \log_2 \chi^*(-\Sigma) \rceil + 1$.

%%===================================================%%
\section{A Catalog of Examples}\label{x}\

We apply the theory to some of the simpler ways one can derive a signed graph from an ordinary graph.  This gives us several general examples.  

The chromatic polynomial and the chromatic number of $\Gamma$ are, respectively, $\chi_\Gamma(\lambda)$ and $\chi(\Gamma)$.
For the geometrical aspects we need the standard basis vectors of $\bbR^n$: 
$$\B_1=(1,0,\ldots,0),\ \B_2=(0,1,0,\ldots,0),\ \ldots,\ \B_n=(0,\ldots,0,1).$$

%%%%%%
\subsection{Full signed graphs}\label{full}\

In this example 
\begin{quote}
\quad $\Sigma$ is a signed graph with no half or loose edges or negative loops, 
\par $\Sigma\full$ is $\Sigma$ with a half edge at every vertex, and 
\par $\Sigma^\circ$ is $\Sigma$ with a negative loop at every vertex.  
\end{quote}
Whether a half edge or negative loop is added makes no almost difference, because each is an unbalanced edge.  Write $f_i$ for the unbalanced edge added to $v_i$.
\begin{itemize}
\item \emph{Balance}: The balanced subgraphs in $\Sigma\full$ are the same as those of $\Sigma$.

\item \emph{Closed sets}: An edge set in $\Sigma\full$ is closed if and only if it consists of the induced edge set $E(\Sigma\full){:}W$ together with a balanced, closed subset of $E(\Sigma){:}W^c$, for some vertex set $W \subseteq V$.  $\Sigma^\circ$ is similar.  
\par Different choices of $W$ give different edge sets (which may not be the case when the signed graph is not full).

\item \emph{Vectors}: 
\par $\bx(E(\Sigma\full))$ adds to $\bx(E(\Sigma))$ the unit basis vectors $\B_i$.  
\par $\bx(E(\Sigma^\circ))$ adds the vectors $2\B_i$ instead.  (That is the only difference we see in the two kinds of full graph.)

\item \emph{Incidence matrix}:  The columns are the vectors $\bx(E)$.  Thus, 
\par $\Eta(\Sigma\full)$ is $\Eta(\Sigma)$ with the columns of an identity matrix $I_n$ adjoined.
\par $\Eta(\Sigma^\circ)$ is $\Eta(\Sigma)$ with the columns of $2I_n$ adjoined.

\item \emph{Hyperplane arrangement}:  $\cH[\Sigma\full] = \cH[\Sigma^\circ]$, and both equal $\cH[\Sigma]$ together with all the coordinate hyperplanes $x_i = 0$. 

\item \emph{Chromatic polynomials}: 
\par $\chi^*_{\Sigma\full}(\lambda) = \chi^*_{\Sigma^\circ}(\lambda) = \chi^*_\Sigma(\lambda)$.  \par
$\chi_{\Sigma\full}(\lambda) = \chi_{\Sigma^\circ}(\lambda) = \chi^*_\Sigma(\lambda-1)$ by Theorem \ref{T:expansion}, since the only stable set is $W = \eset$.

\item \emph{Chromatic numbers}:  $\chi(\Sigma\full) = \chi(\Sigma^\circ) = \chi^*(\Sigma)$ since the unbalanced edges prevent the use of color 0.
\end{itemize}

%%%%%%
\subsection{All-positive signed graphs}\label{allpos}\

In this example we assume $\Gamma$ is a graph with no half or loose edges.  $+\Gamma$ has almost exactly the same properties as its underlying graph.  
\begin{itemize}
\item \emph{Balance}: Every subgraph is balanced.  $b(S) = c(S)$ for all $S \subseteq E$.

\item \emph{Closed sets}: $S$ is closed $\iff$ every edge whose endpoints are connected by $S$ is in $S$. 
\par Closure in $+\Gamma$ is identical to the usual closure in $\Gamma$, and the closed sets in $+\Gamma$ are the same as in $\Gamma$.

\item \emph{Vectors}:  If $e$ has endpoints $v_i,v_j$, then $\bx(e) = \pm(\B_j-\B_i)$. 
\par All $\bx(e) \in$ the subspace $x_1+\cdots+x_n=0$.  

When $\Gamma = K_n$, if one takes both signs the set of vectors is the classical root system 
$$A_{n-1} := \{ \B_j - \B_i : i,j \leq n,\ i \neq j \}.$$
Thus, $\bx(E)$ for any graph is a subset of $A_{n-1}$.

\item \emph{Incidence matrix}:  $\Eta(+\Gamma)$ is the `oriented incidence matrix' of $\Gamma$.

\item \emph{Hyperplane arrangement}:  If $e_k$ has endpoints $v_i,v_j$, then $h_k$ has equation $x_i=x_j$. 
\par All hyperplanes $h_k$ contain the line $x_1=\cdots=x_n$.  

Take $\Gamma = K_n$; then $\cH[+K_n] = \cA_{n-1}$, the hyperplane arrangement dual to $A_{n-1}$.

\item \emph{Chromatic polynomials}: $\chi_{+\Gamma}(\lambda) = \chi^*_{+\Gamma}(\lambda) = \chi_{\Gamma}(\lambda)$, the chromatic polynomial of $\Gamma$.

\item \emph{Chromatic numbers}:  $\chi(+\Gamma) = \lfloor\tfrac12\chi(\Gamma)\rfloor$ and $\chi^*(+\Gamma) = \lceil\tfrac12\chi(\Gamma)\rceil$. 
\end{itemize}

%%%%%%
\subsection{All-positive, full signed graphs}\label{allposfull}\

The signed graph $+\Gamma\full$ is closely related to $\Gamma+v_0$, which consists of $\Gamma$ and an extra vertex $v_0$ adjacent to all of $V$ by edges $v_0v_i$.  There is a natural bijection $\alpha: E(+\Gamma\full) \to E(\Gamma+v_0)$ by $\alpha(e) := e$ if $e \in E(\Gamma)$ and $\alpha(f_i) := v_0v_i$.
\begin{itemize}
\item \emph{Balance}: $S$ is balanced if and only if $\alpha(S)$ does not contain any edges at $v_0$.

\item \emph{Closed sets}: $S$ is closed $\iff$ $\alpha(S)$ is closed in $\Gamma+v_0$.

\item \emph{Chromatic polynomials}: $\chi_{+\Gamma\full}(\lambda) = \chi_{\Gamma+v_0}(\lambda) = \chi_{\Gamma}(\lambda-1)$.

\item \emph{Chromatic numbers}:  $\chi(+\Gamma\full) = \lfloor\tfrac12\chi(\Gamma)\rfloor$ and $\chi^*(+\Gamma\full) = \lceil\tfrac12\chi(\Gamma)\rceil$. 
\end{itemize}

%%%%%%
\subsection{All-negative signed graphs}\label{allneg}\

Again assume $\Gamma$ is a graph with no unbalanced edges.  $-\Gamma$ is quite interesting.  
\begin{itemize}
\item \emph{Balance}: 
\par A subgraph is balanced $\iff$ it is bipartite. 
\par $b_{-\Gamma}(S) =$ the number of bipartite components of $S$ (including isolated vertices).

\item \emph{Closed sets}: $S$ is closed if and only if the union of its non-bipartite components is an induced subgraph.

\item \emph{Vectors}:  If $e$ has endpoints $v_i,v_j$, then $\bx(e) = \B_i+\B_j$ (or its negative).

\item \emph{Incidence matrix}:  $\Eta(-\Gamma)$ is the `unoriented incidence matrix' of $\Gamma$.

\item \emph{Hyperplane arrangement}: 
\par $h_k$ has equation $x_i+x_j=0$ if $e_k$ has endpoints $v_i,v_j$.  
%\par The number of regions is 
$$r(\cH[-\Gamma]) = \sum_{F\in\Lat\Gamma} \sum_{W \subseteq V(F)^c: \text{stable}} (-1)^{n-c(F)+|W|} |\chi_{(\Gamma\setminus W)/F}(-1)|.$$

\item \emph{Chromatic polynomials}: 
\begin{align*}
\chi^*_{-\Gamma}(\lambda) &= \sum_{F\in\Lat\Gamma} \chi_{\Gamma/F}\big(\frac12\lambda\big) \quad \text{(Zaslavsky (1982c), Theorem 5.2). }\\
\chi_{-\Gamma}(\lambda) &= \sum_{F\in\Lat\Gamma}  \sum_{W \subseteq V(F)^c: \text{stable}} \chi_{(\Gamma\setminus W)/F}\big(\frac{\lambda-1}{2}\big).
\end{align*}

\item \emph{Chromatic numbers}: 
\par $\chi^*(-\Gamma) =$ the largest size of a matching in the complement of $\Gamma$ (based on Zaslavsky (1982c), page 299). 
\par $\chi(-\Gamma)$ has not yet seemed interesting.
\end{itemize}

%%%%%%
\subsection{Signed expansion graphs}\label{expansion}\

Now assume we have a simple graph $\Gamma$.  The properties of $\pm\Gamma$ and $\pm\Gamma\full$ are closely related to those of $\Gamma$.  
\begin{itemize}
\item \emph{Balance}:  Each set $S \subseteq E(\Gamma)$ gives $2^{n-c(S)}$ balanced subsets of $E(\pm\Gamma)$ by switching $+S$.

\item \emph{Closed sets}:  Each closed set $S \subseteq E(\Gamma)$ gives $2^{n-c(S)}$ balanced closed subsets of $E(\pm\Gamma)$ by switching $+S$. 
\par Each unstable (i.e., non-independent) vertex subset $W \subseteq V$ gives $2^{n-|W|-c(S)}$ unbalanced closed sets for each closed set $S$ in $\Gamma \setm W$ by taking $E(\pm\Gamma){:}W \cup S'$ where $S'$ is any switching of $+S$ in $\Gamma \setm W$.

%\item \emph{Vectors}:  

\item \emph{Hyperplane arrangement}:  The numbers of regions are 
$$r(\cH[\pm\Gamma\full]) = 2^n (-1)^n\chi_\Gamma(-1) = 2^n |\chi_\Gamma(-1)|$$ 
and 
$$
r(\cH[\pm\Gamma]) = \sum_{W \subseteq V: \text{ stable in }\Gamma} (-2)^{n-|W|} |\chi_{\Gamma\setm W}(-1)|.
$$

\item \emph{Chromatic polynomials}: 
$$\chi_{\pm\Gamma\full}(\lambda) = 2^n \chi_\Gamma(\tfrac12(\lambda-1)),$$ 
$$\chi^*_{\pm\Gamma}(\lambda) = 2^n \chi_\Gamma(\tfrac12\lambda),$$
and 
$$
\chi_{\pm\Gamma}(\lambda) = \sum_{W \subseteq V: \text{ stable in }\Gamma} 2^{n-|W|} \chi_{\Gamma\setm W}(\tfrac12(\lambda-1)).
$$

\item \emph{Chromatic numbers}: 
\par $\chi(\pm\Gamma\full) = \chi^*(\pm\Gamma) = \chi(\Gamma)$ because, if $\Gamma$ is properly colored by the set $\{1,\ldots,\chi(\Gamma)\}$, the colors $i$ can be transferred directly to colors $+i$ (or $-i$ if desired), giving a proper coloration of $\pm\Gamma$.
\par $\chi(\pm\Gamma) = \chi(\Gamma)-1$ because the color $0$ can be substituted for the color $\chi(\Gamma)$.

\end{itemize}

%%%%%%
\subsection{Complete signed expansion graphs}\label{completex}\

The signed expansions $\pm K_n$, called the \emph{complete signed link graph}, and $\pm K_n\full$ (which we now define to have any choice of a half edge or negative loop $f_i$ at each vertex), called the \emph{complete signed graph}, have elegantly simple properties.
\begin{itemize}

\item \emph{Closed sets}: For the complete signed graph the lattice of closed sets, $\Lat(\pm K_n\full)$, is isomorphic to the lattice of signed partial partitions of $V$ (Dowling (1973b)).

\item \emph{Vectors}:  
$$\bx(E(\pm K_n)) = \{ \pm(\B_j-\B_i), \pm(\B_j+\B_i) : i \neq j\}$$ 
where we take either $+$ or $-$ for each vector, and    
$$\bx(E(\pm K_n\full)) = \{ \pm(\B_j-\B_i), \pm(\B_j+\B_i) : i \neq j\} \cup \{ \pm \B_i \}$$ 
if every $f_i$ is a half edge (but take $\pm2\B_i$ instead for an $f_i$ that is a negative loop) where again we take either $+$ or $-$ for each vector.

If we take both signs for each vector we get the classical root systems 
$$D_n := \{ \pm(\B_j-\B_i), \pm(\B_j+\B_i) : i \neq j\}$$ 
from $\pm K_n$ (where we take both $+$ and $-$ signs), and 
$$B_n := D_n \cup \{ \pm \B_i \} \text{ and } C_n := D_n \cup \{ \pm2\B_i \}$$ 
from $\pm K_n\full$ (the former if all $f_i$ are half edges, the latter if they are negative loops).

\item \emph{Hyperplane arrangement}: 
$$\cH[\pm K_n\full] = \cB_n = \cC_n \text{ and } \cH[\pm K_n] = \cD_n,$$
the duals of $B_n$, $C_n$, and $D_n$.  The numbers of regions are $2^n n!$ and $2^{n-1} n!$, respectively.

\item \emph{Chromatic polynomials}: 
\begin{align*}
\chi_{\pm K_n\full}(\lambda) &= (\lambda-1)(\lambda-3)(\lambda-5)\cdots(\lambda-2n+1), \\
\chi_{\pm K_n}(\lambda) &= (\lambda-1)(\lambda-3)(\lambda-5)\cdots(\lambda-2n+3)\cdot(\lambda-n+1), \\
\chi^*_{\pm K_n}(\lambda) = \chi^*_{\pm K_n\full}(\lambda) &= \lambda(\lambda-2)(\lambda-4)\cdots(\lambda-2n+2).
\end{align*}

\item \emph{Chromatic numbers}:  
$$\chi(\pm K_n\full) = \chi^*(\pm K_n\full) = \chi^*(\pm K_n) = n$$ and $$\chi(\pm K_n) = n-1.$$

\end{itemize}

%%===================================================%%
%\newpage
\section{Line Graphs}\label{lg}

Several definitions exist for a line graph of a signed graph, which in different ways sign the edges of the line graph of the underlying graph.  Two seem (to me) well motivated.  They are the $\times$-line signed graph of Mukti Acharya (1982a, 2009a), which closely follows the combinatorial spirit of a line graph, and my definition based on the matrix properties of line graphs (Zaslavsky 1984c, 2010b).  It is that last which connects to geometry.

In this section all our graphs are link graphs, but not necessarily simple.  (Allowing loops and half edges adds much complexity.)  

The \emph{line graph} of an unsigned graph $\Gamma=(V,E)$ is the graph $\Lambda(\Gamma)$ of adjacency of edges in $\Gamma$.  Its vertices are the edges of $\Gamma$, and two edges are adjacent if they have a common endpoint.  When $e, f \in E(\Gamma)$ are parallel, in $\Lambda(\Gamma)$ they are doubly adjacent.  
%(For a technically correct definition we could let $$V(\Lambda(\Gamma)) := E(\Gamma) \text{ and } E(\Lambda(\Gamma)) := \{ \{e,f,v\} : e,f \in E(\Gamma) \text{ and } v \in V(e) \cap V(f) \},$$ which permits us to distinguish the two edges $ef$ generated by parallel $e$ and $f$ as (say) $ef_v$ and $ef_w$, when such distinction is needed.)
\smallskip

\begin{center}
\quad $\Gamma$ \quad\qquad\qquad\qquad \qquad \ \qquad\qquad $\Lambda(\Gamma)$\\[10pt]

\includegraphics[scale=.8]{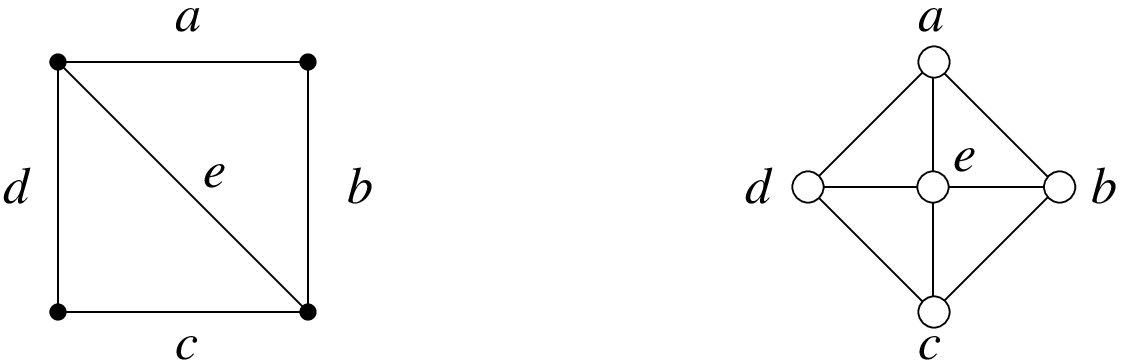}	
\end{center}

Construction of the line graph of a signed graph $\Sigma$ has to be approached through an orientation of $\Sigma$.  Thus, we begin with line graphs of bidirected graphs, which are identical (except for the name) with line graphs of oriented signed graphs.

%%%%%%%%%%%%
\subsection{Bidirected line graphs and switching classes}\label{lgb}\

\subsubsection{Bidirected graphs}\

The \emph{line graph of a bidirected graph $\Beta$} is a bidirection of the line graph of $|\Beta|$.  We write $\Lambda(\Beta) := (\Lambda(|\Beta|), \tau_\Lambda)$, where $\tau_\Lambda$ is the bidirection.  To define $\tau_\Lambda(ef)$, where $ef \in E(\Lambda(|\Beta|))$, let $v$ be the vertex at which $e, f$ are adjacent.  Then we define 
$$\tau_\Lambda(e,ef) := \tau(v,e).$$

\subsubsection{Signed graphs}\

This definition implies that, given a signed graph $\Sigma$, to define a line graph we must first orient $\Sigma$ as $\Beta$, then take the line graph $\Lambda(\Beta)$.  

\noindent\parbox{2.5in}{\quad An oriented signed graph and its oriented line graph.}
\hfill
\raisebox{-1.5cm}
{\includegraphics[scale=.8]{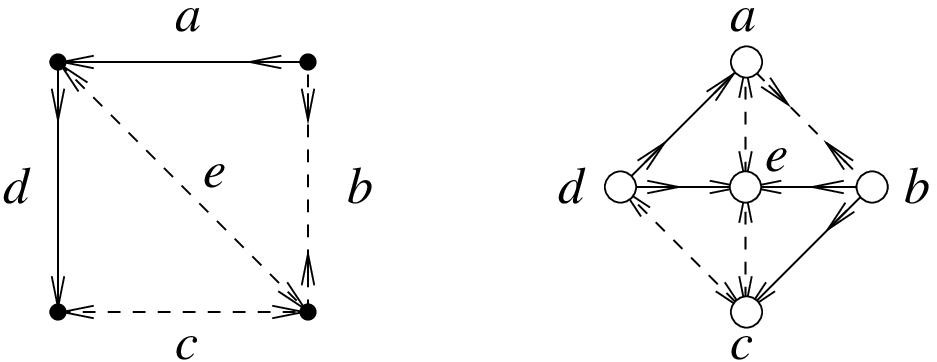}}

Different orientations of $\Sigma$ give different bidirected line graphs $\Lambda(\Beta)$, which may have different signed graphs $\Sigma_{\Lambda(\Beta)}$.  Indeed, reorienting an edge of a bidirected graph corresponds to switching the corresponding vertex in its line graph.  Switching $\Sigma$ itself has a more complicated effect.

\begin{lem}\label{L:lgsgsw}
Any orientations of any two switchings of\/ $\Sigma$ have line graphs that are switching equivalent.
\end{lem}

\begin{proof}
We assume there are no parallel edges; the proof is not much different if there are any.

Let $\Sigma^\zeta$ be a switching of $\Sigma$ and let $\tau$ and $\tau'$ be orientation functions of $\Sigma$ and $\Sigma^\zeta$, respectively, giving bidirected graphs $\Beta$ and $\Beta'$ on the underlying graph $\Gamma := |\Sigma|$.  Then $\tau(v,e)\tau(w,e) = -\sigma(e)$ and $\tau'(v,e)\tau'(w,e) = -\zeta(v)\sigma(e)\zeta(w)$ for each edge $e$ with $V(e) = \{v,w\}$.  

Let $\Lambda := \Lambda(\Beta)$ and $\Lambda' := \Lambda(\Beta')$; they have the same underlying graph $\Lambda(\Gamma)$.  Suppose $e,f$ are adjacent at $v$.  In the line graph, $\tau_\Lambda(e,ef) = \tau(v,e)$.  Thus, $$\sigma_\Lambda(ef) = -\tau_\Lambda(e,ef)\tau_\Lambda(f,ef) = -\tau(v,e)\tau(v,f)$$ and, similarly, $\sigma'_\Lambda(ef) = -\tau'(v,e)\tau'(v,f)$.

Let $W := e_0e_1\cdots e_{l-1}e_l$, where $e_l = e_0$, be a closed walk in $\Lambda$.  Thus, $e_{i-1}, e_i$ have a common vertex $v_i$ in $\Gamma$.  Then 
\begin{equation}
\begin{aligned}
\sigma_\Lambda(W) &= \sigma_\Lambda(e_0e_1)\cdots\sigma_\Lambda(e_{l-1}e_l) \\
&= \big[ -\tau(v_1,e_0)\tau(v_1,e_1) \big] \big[ -\tau(v_2,e_1)\tau(v_2,e_2) \big] \cdots \big[ -\tau(v_l,e_{l-1})\tau(v_l,e_l) \big] \\
&= (-)^{l} \tau(v_1,e_0) \big[\tau(v_1,e_1)\tau(v_2,e_1)\big] \cdots \big[\tau(v_{l-1},e_{l-1})\tau(v_l,e_{l-1})\big] \tau(v_l,e_l) .
\end{aligned}
\label{E:Wsign}
\end{equation}

Now there are two cases.

If all $v_i = v_1$, then $\sigma_\Lambda(W) = (-)^{l} \tau(v_1,e_0)\tau(v_l,e_l) = (-)^l$ since $e_0=e_l$ and $v_l=v_1$.

Otherwise, not all $v_i$ are the same vertex.  A consecutive pair $v_{i-1}, v_i$ may be the same or different.  
If they are the same, the factor $[\tau(v_{l-1},e_{i-1})\tau(v_i,e_{i-1})] = +$, and also $W' := e_0e_1\cdots e_{i-2}e_i \cdots e_{l-1}e_l$ is a walk in $\Lambda$.  Then $\sigma_\Lambda(W) = -\sigma_\Lambda(W')$.  In this way we can reduce $W$ by eliminating consecutive equal vertices while negating the sign of the walk.  Similarly, if $v_1 = v_l$ we can eliminate $v_l$ and $e_l$ from the reduced walk.  Let $W'' = f_0f_1\cdots f_m$ be the walk in $\Lambda$ that results after all these reductions and let $w_i$ be the common vertex of $f_{i-1}, f_i$.  $W''$ has positive length and $f_0=f_m$, so $W''$ is a closed walk and it has sign $(-)^{l-m}\sigma_\Lambda(W)$.  Furthermore, $V(f_i) = \{w_i,w_{i+1}\}$ for $0<i<m$.  Define $w_0$ so that $V(f_0) = \{w_0,w_1\}$.  Now $w_0f_0w_1f_1\cdots f_{m-1}w_m$ is a walk in $\Gamma$.  Because $f_0=f_m$ and, by the construction of $W''$, $w_1 \neq w_m$, it must be true that $w_0=w_m$.  Therefore, $W_0 := w_0f_0w_1\cdots f_{m-1}w_m$ is a closed walk of length $m$ in $\Gamma$.  
Now we evaluate $\sigma_\Lambda(W'')$.  From \eqref{E:Wsign},
\begin{equation*}
\begin{aligned}
\sigma_\Lambda(W'') &= \sigma_\Lambda(f_0f_1)\cdots\sigma_\Lambda(f_{m-1}f_m) \\
&= (-)^{m} \tau(w_1,f_0) 
\\ &\qquad \cdot 
\big[\tau(w_1,f_1)\tau(w_2,f_1)\big] \cdots \big[\tau(w_{m-1},f_{m-1})\tau(w_m,f_{m-1})\big] \tau(w_m,f_m) \\
&= (-)^{m} \tau(w_0,f_0) \tau(w_1,f_0) 
\\ &\qquad \cdot \big[\tau(w_1,f_1)\tau(w_2,f_1)\big] \cdots \big[\tau(w_{m-1},f_{m-1})\tau(w_m,f_{m-1})\big] \\
&= \sigma(W_0) .
\label{E:W''sign}
\end{aligned}
\end{equation*}

We conclude that  
\begin{equation}
\sigma_\Lambda(W) = (-)^{l-m}\sigma(W_0)
\label{E:lgwalksign}
\end{equation}
when not all the vertices $v_i$ are the same vertex.  
When all $v_i$ are the same, we can take $W_0$ to be a trivial walk (length $m=0$) and once again we have the same formula \eqref{E:lgwalksign}.

We prove the lemma by observing that $\sigma(W_0)$ and hence also $\sigma(W)$ are not affected by the choice of orientation and are not altered by switching $\Sigma$.  Therefore $\Lambda$ and $\Lambda'$ have the same positive circles.  By Proposition \ref{P:switchingequiv}(ii), $\Lambda$ and $\Lambda'$ are switching equivalent. 
\end{proof}

The line graph of a signed graph $\Sigma$ cannot be a signed graph, because reorienting an edge switches the corresponding vertex in the line graph.  Therefore, $\Lambda(\Sigma)$ must be a switching class of signatures of $\Lambda(|\Sigma|)$.

\begin{thm}\label{T:lgswclass}
The line graph of a switching class of signed graphs is a well defined switching class of signed graphs.
\end{thm}

\begin{proof}
The theorem means that if two signed graphs are switching equivalent, and if each one is oriented arbitrarily, the signed graphs of the line graphs of the two oriented signed graphs are switching equivalent.  That is Lemma \ref{L:lgsgsw}.
\end{proof}

In view of this theorem we may write 
\begin{quote}
\begin{itemize}
\item[{$\Lambda[\Sigma] :=$}]
 the switching class of line graphs of the signed graphs in the switching class $[\Sigma]$.  
\end{itemize}
\end{quote}
I sometimes refer to \emph{a line graph} of $\Sigma$, meaning any signed graph in the switching class $\Lambda[\Sigma]$.  

%%%%%%
\subsubsection{All-negative signatures and their line graphs}\

There is one circumstance in which there is a well defined signed line graph: an all-negative signature.  

\begin{prop}\label{P:allneg}
If\/ $\Gamma$ is a link graph, then $\Lambda[-\Gamma] = [-\Lambda(\Gamma)]$.
\end{prop}

\begin{proof}
Orient $-\Gamma$ so every edge is extraverted; that is, $\tau(v,e) \equiv +$.  Then in $\Lambda(-\Gamma,\tau)$, every edge is extraverted; thus, the signed graph underlying $\Lambda(-\Gamma,\tau)$ has all negative edges.
\end{proof}

Consequently, we can say that $$\Lambda(-\Gamma) = - \Lambda(\Gamma),$$ as in the following example:

\noindent\parbox{3in}{\quad A line graph of an all-negative signed graph, oriented with extraverted edges.}
\hfill
\raisebox{-1.5cm}
{\includegraphics[scale=.8]{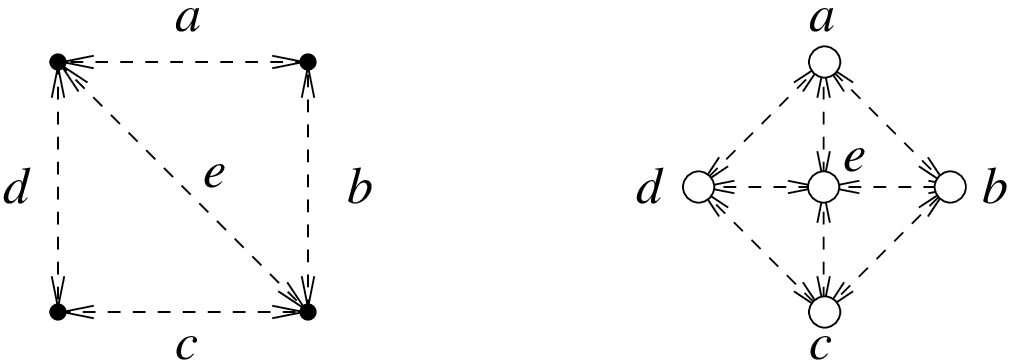}}	

%%%%%%
\subsubsection{All-positive signatures and their line graphs}\

On the other hand, if $\Sigma$ is all positive, its line graph cannot usually be made to be all positive or all negative.  Thus, all-negative signed graphs are special.  Indeed, in connection with line graphs the best way to think of an ordinary graph $\Gamma$ is as $-\Gamma$, not $+\Gamma$ as in most other respects.

Nonetheless there is value in looking into all-positive graphs, because their orientations are digraphs.  The Harary--Norman line graph of a digraph $D$ \cite{HN} is precisely the positive part of the line graph of $D$, if $D$ is treated as an oriented all-positive signed graph:  $\Lambda_{\textrm{HN}}(D) = \Lambda(+D)^+$.  $\Lambda_{\textrm{HN}}(D)$ detects directed paths of length two in $D$ but ignores head-to-head and tail-to-tail adjacencies.  $\Lambda(+D)$ records all adjacencies.  (Previous attempts to encompass all edge adjacencies were handicapped by not having bidirected edges in the line graph.)

%%%%%%%%%%%%
\subsection{Adjacency matrix and eigenvalues}\label{lga}\

With line-graph matrices we come to the third reason the incidence matrix is important.  The adjacency matrix of the line graph of an ordinary graph is computed directly from the incidence matrix of the graph, and the same holds true for signed graphs.  

An example, using a particular choice of orientation of $\Sigma_4$:\\
\begin{center}
\includegraphics[scale=.7]{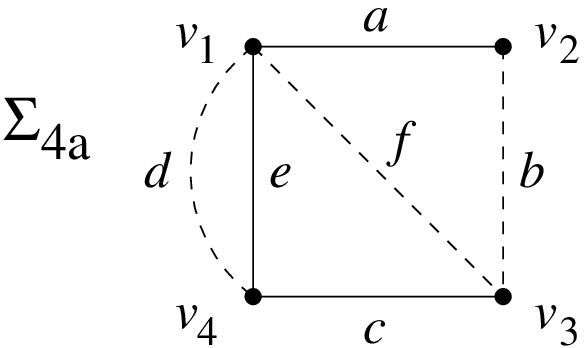}
\qquad
\raisebox{1.2cm}{
$A(\Lambda(\Sigma_{4\mathrm{a}})) = 
\begin{pmatrix}
0	&1	&0	&-1	&-1	&1	\\
1	&0	&-1	&0	&0	&-1	\\
0	&-1	&0	&-1	&1	&-1	\\
-1	&0	&-1	&0	&0	&0	\\
-1	&0	&1	&0	&0	&1	\\
1	&-1	&-1	&0	&1	&0	\\
\end{pmatrix}$
}
\end{center}

\begin{thm}\label{T:lga}
For a bidirected link graph\/ $\Sigma$, $A(\Lambda(\Sigma)) = 2I - \Eta(\Sigma)\transpose \Eta(\Sigma)$.
\end{thm}

\begin{proof}
The $(j,j)$ entry of $\Eta(\Sigma)\transpose \Eta(\Sigma)$ is the sum over all vertices of $\eta_{v_ie_j}^2 = 1$, therefore it equals $2$.

The $(j,k)$ entry of $\Eta(\Sigma)\transpose \Eta(\Sigma)$ for $j \neq k$ is the sum over all vertices of $\eta_{v_ie_j}\eta_{v_ie_k}$.  By Equation \eqref{E:taueta} this is 0 if $e_j$ and $e_k$ are not adjacent, and if they are adjacent at $v_m$ then it is $\tau(v_i,e_j)\tau(v_i,e_k) = -\sigma(e_je_k)$.  

Thus, the off-diagonal entries of $\Eta(\Sigma)\transpose \Eta(\Sigma)$ are those of $-A(\Lambda(\Sigma))$ and the diagonal entries all equal 2.
\end{proof}

The orientation of $\Sigma$ used to calculate $\Lambda(\Sigma)$ can affect the values in $A(\Lambda(\Sigma))$, as the choice of orientation switches the line graph, and that corresponds to conjugating $A(\Lambda(\Sigma))$ by a diagonal matrix with $\pm1$'s on the diagonal (see the proof of Theorem \ref{T:e}).  However, the eigenvalues of $A(\Lambda(\Sigma))$ are independent of the choice of orientation.

We can interpret Theorem \ref{T:lga} as saying that the inner product of representation vectors $\bx(e_j)$ and $\bx(e_k)$ equals 2 if $j=k$ and $-\sigma(e_je_k)$ if $j \neq k$.  A matrix of inner products is known as a \emph{Gram matrix}; thus, $2I - A(\Lambda(\Sigma))$ is a Gram matrix of vectors with length $\sqrt2$.

\begin{cor}\label{C:lgebound}
All the eigenvalues of a line graph of a signed graph are $\leq 2$.
\end{cor}

\begin{proof}
Rewrite Proposition \ref{T:lga} as $2I - A(\Lambda(\Sigma)) = \Eta(\Sigma)\transpose \Eta(\Sigma)$.  A matrix of the form $M\transpose M$ has non-negative real eigenvalues.
\end{proof}

In unsigned graph theory the eigenvalues of a line graph are $\geq -2$.  Corollary \ref{C:lgebound} is the generalisation to signed graphs, because in what concerns line graphs, an unsigned graph should be taken as all negative, and the eigenvalues of $-\Sigma$ are the negatives of those of $\Sigma$ (since $A(-\Sigma) = -A(\Sigma)$).

%%%%%%%%%%%%
\subsection{Reduced line graphs and induced non-subgraphs}\label{lgri}\

If $\Sigma$ has a negative digon, that is, a pair of parallel edges $e,f$, one positive and the other negative, then in $\Lambda(\Sigma)$ there is a double edge $ef$ that forms a negative digon.  Therefore, the $(e,f)$ entry of $A(\Lambda(\Sigma))$ equals 0 and, correspondingly, in the reduced line graph $\bar\Lambda(\Sigma)$ the vertices $e$ and $f$ are not adjacent.  
My conclusion:  In what regards adjacency matrices and eigenvalues, one should look at reduced line graphs rather than unreduced line graphs; but look at the (reduced) line graphs of both unreduced and reduced signed graphs!

A well known theorem of Beineke and Gupta from around 1970 is that a simple graph is a line graph if and only if it has no induced subgraph that is one of nine particular graphs, all of order at most 6.  Chawathe and Vijayakumar (1990a) found the analogous 49 excluded induced switching classes, all of order at most 6, for signed simple graphs that are reduced line graphs of signed graphs.  
(I believe the value 6 for the largest order is due to automorphism properties of the classical root systems.)

%%===================================================%%
\section{Angle Representations}\label{angle}

In this section we represent a signed graph by mapping the vertices, instead of the edges, to vectors.  In such `vertex representations' it is best to assume all underlying graphs are simple.  
For a non-zero vector $\by$, the unit vector in the same direction is $\hat\by := \|\by\|\inv\by$.  

An \emph{angle representation} of $\Sigma$ is a mapping $\brho:  V \to \bbR^d$, for some dimension $d$, such that 
$$
\hat\brho(v) \cdot \hat\brho(w) = \frac{a_{vw}}{\nu} = \begin{cases}
\qquad 0, &\text{ if $vw$ is not an edge and $v \neq w$,}  \\
+1/\nu, &\text{ if $vw$ is a positive edge, and}  \\
-1/\nu, &\text{ if $vw$ is a negative edge,}  \\
\end{cases}
$$
for a positive constant $\nu$.  Equivalently, the representing vectors $\brho(v), \brho(w)$ of vertices $v,w$ make an angle 
$$
\angle\big(\brho(v), \brho(w)\big) = \begin{cases}
\theta = \arccos(1/\nu) \in [0,\pi/2], &\text{ if } \sigma(vw) = +, \\
\pi-\theta, &\text{ if } \sigma(vw) = -, \\
\pi/2, &\text{ if } vw \notin E, \ v \neq w.  
\end{cases}
$$
When $X \subseteq \bbR^d$, we call $\brho$ an \emph{angle representation in $X$} if $\Im \brho \subseteq X$.  
As the length of $\brho(v)$ has no role in the definition, one still has an angle representation after multiplying any $\brho(v)$ by any positive real number.  Thus, for instance, one may assume all the representing vectors have a particular desired length such as 1 or 2.  

Switching $v$ in $\Sigma$ corresponds to replacing $\brho(v)$ by $-\brho(v)$ in the angle representation.

A generalization of the Gram-matrix (that is, dot-product) interpretation of Theorem \ref{T:lga} is a \emph{Gramian angle representation} of $\Sigma$.  That is an angle representation such that 
$$\brho(v) \cdot \brho(w) = a_{vw}$$ 
for every pair of distinct vertices.  
It follows by comparing the two definitions that $\|\brho(v)\|\cdot\|\brho(w)\| = \nu$ for adjacent vertices.  

An \emph{anti-Gramian angle representation} of $\Sigma$ is a Gramian angle representation of $-\Sigma$.  Vijayakumar uses anti-Gramian representations (see Vijayakumar (1987a) et al.).  Example \ref{X:lgrepn} will show why one wants them.

\begin{prop}\label{P:gramian}
In a Gramian angle representation $\brho$ of a connected signed simple graph $\Sigma$:
\begin{enumerate}[{\rm(a)}]
\item If\/ $\Sigma$ is not bipartite, all representing vectors $\brho(v)$ have the same length $\sqrt\nu$.
\item If\/ $\Sigma$ is bipartite with color classes $V_1$ and $V_2$, then $\|\brho(v)\| = \alpha$ if\/ $v \in V_1$ and $\|\brho(v)\| = \nu/\alpha$ if\/ $v \in V_2$, where $\alpha>0$.  
Then $\brho'$ defined by $\brho'(v) = \hat\brho(v)\sqrt\nu$ is an angle representation in which all representing vectors have the same length.
\end{enumerate}
\end{prop}

\begin{proof}[Idea of Proof]
Apply the equation $\|\brho(v)\| \|\brho(w)\| = \nu$ for an edge $vw$, propagated around an odd circle if there is one, and an even circle if there is not.
\end{proof}

In a \emph{normalized} Gramian angle representation all vectors have the same length.  Then the Gram matrix of the representing vectors is $A(\Sigma)+\nu I$.  By Proposition \ref{P:gramian} any Gramian angle representation becomes normalized if we replace $\brho(v)$ by $\sqrt\nu\hat\brho(v)$.  Henceforth we assume all Gramian representations are normalized.  

\begin{thm}\label{T:grame}
A signed simple graph $\Sigma$ has a Gramian (or, anti-Gramian) angle representation with constant $\nu$ if and only if the eigenvalues of\/ $\Sigma$ are $\geq -\nu$ (respectively, $\leq \nu$).
\end{thm}

\begin{proof}
This proof is based on the treatment of equiangular lines by Seidel et al.\ (see, e.g., Seidel (1976a, 1995a) or Godsil and Royle (2001a)).  

We consider a normalized Gramian angle representation.  The Gram matrix $A(\Sigma)+\nu I$ has an eigenvalue $\lambda+\nu$ for each eigenvalue $\lambda$ of $A(\Sigma)$.  As a Gram matrix has non-negative eigenvalues, every $\lambda \geq -\nu$.

Now assume $\Sigma$ has eigenvalues $\geq -\nu$.  The matrix $A(\Sigma)+\nu I$ is positive semidefinite and symmetric.  It follows by matrix theory that $A(\Sigma)+\nu I$ is the Gram matrix of vectors $\bv_i \in \bbR^n$ for $v_i \in V$, i.e., $a_{ij} + \nu\delta_{ij} = \bv_i\cdot\bv_j$ for all $i,j$.  Then $\brho(v_i) := \bv_i$ is a normalized Gramian angle representation of $\Sigma$ with constant $\nu$.
\end{proof}

\begin{exam}\label{X:lgrepn}
The mapping $\bx: E(\Sigma) \to \bbR^n$ of Section \ref{vectors}, which gives a vector representation of $\Sigma$, gives an anti-Gramian angle representation of $\bar\Lambda(\Sigma)$.  We take $\brho := \bx$, since $V(\bar\Lambda(\Sigma)) = E(\Sigma)$.  The constant is $\nu = 2$ and the angle is $\theta = \pi/3$.  
Every vector $\bx(e)$ has the same length, $\sqrt2$, and the inner products are $+1$ if $\sigma_\Lambda(ef) = -$, in which case the angle between $\bx(e)$ and $\bx(f)$ is $\pi/3$, and $-1$ if $\sigma_\Lambda(ef) = +$, in which case the angle between $\bx(e)$ and $\bx(f)$ is $2\pi/3$.  (The signs reverse because the representation is anti-Gramian.)

The vectors $\bx(e)$ are some of the vectors of the root system $D_n$ mentioned in Section \ref{completex}.  The image of the representation of $\bar\Lambda(\pm K_n)$ is all of $D_n$.  The treatment of $\bx$ as an angle representation of a reduced line signed graph is implicit in Cameron, Goethals, Seidel, and Shult (1976a), but explicit line graphs of signed graphs only came later, in Zaslavsky (1979a, 1984c, 2010b). 
\end{exam}

The root system $E_8$ is defined by 
$$
E_8 := D_8 \cup \big\{ \tfrac12(\eps_1,\ldots,\eps_8) \in \bbR^8 : \eps_i \in \{\pm1\},\ \eps_1\cdots\eps_8 = +1 \big\}.
$$

\begin{thm}\label{T:anglelg}
An anti-Gramian angle representation of\/ $\Sigma$ with $\nu=2$ is a vector representation of a reduced line graph $\bar\Lambda(\Sigma)$, or else $|V(\Sigma)| \leq 184$ and the representation is in $E_8$.
\end{thm}

\begin{proof}
As Vijayakumar (1987a) observed, Cameron, Goethals, et al.\ (1976a) implies that an anti-Gramian angle representation of $\Sigma$ having $\nu=2$ is, after choosing the appropriate coordinate system, either in $D_n$ for some $n > 0$ or in $E_8$.  If the representation is in $D_n$, then there is a signed graph $\Sigma'$ with vertex set $E(\Sigma)$ whose vector representation $\bx: E \to \bbR^n$ is the same as $\brho$, and it is easy to verify that $\Sigma'$ is a reduced line graph of $\Sigma$.  

If the representation is in $E_8$, the order of $\Sigma$ cannot be greater than the number of pairs of opposite vectors in $E_8$, which is $184$ because $|D_n| = n(n-1)$ and the number of choices for $(\eps_1,\ldots,\eps_8)$ is $2^7$.
\end{proof}

Cameron, Goethals, et al.\ used Gramian angle representations of unsigned graphs to classify the graphs $\Gamma$ whose eigenvalues are $\geq -2$.  They obtained the all-positive and all-negative cases of the preceding theorem.  (The all-positive case corresponds, in our terminology, to a Gramian representation of $-\Gamma$, and the all-negative case to a Gramian representation of $+\Gamma$, since the theorem concerns anti-Gramian representations.)  Then G.R.\ Vijayakumar and his collaborators extended that work to anti-Gramian representations of signed graphs (without line graphs; historically, therefore, there were two independent lines of development treating essentially the same objects: that of line graphs by Zaslavsky and that of angle representations by Vijayakumar et al.).

\begin{cor}\label{T:lgebound}
A signed simple graph has all eigenvalues $\leq 2$ if and only if it is a reduced line graph of a signed graph or it has order $\leq 184$ and has an anti-Gramian angle representation in $E_8$.
\end{cor}

\begin{proof}[Proof of Sufficiency]
Vectors in $D_n$ or $E_8$ have angles $\pi/3$, $2\pi/3$, and $\pi/2$, and any such system of vectors with norm $\sqrt2$ is contained in $D_n$ or $E_8$, therefore an angle representation in $D_n$ or $E_8$ has $\nu=2$.
\end{proof}

Thus, eigenvalues determine whether a signed graph is a line graph, with a finite number of exceptions of explicitly bounded order!  More precisely, the number of signed simple graphs with all eigenvalues $\leq 2$ that are not reduced line graphs of signed graphs is finite and not too large (but not too small either).

We conclude with a description of the crucial example that led to Cameron, Goethals, Seidel, and Shult (1976a) and that also shows why signed graphs are truly the natural domain for line graphs.

\begin{exam}\label{X:glg}
Let $\Gamma$ be a simple graph with $V = \{v_1,\ldots,v_n\}$.  A \emph{cocktail party graph} $\CP_m$ is $K_{2m} \setm M$ where $M$ is a perfect matching.  Hoffman (1977a) defined the \emph{generalized line graph} $\Lambda(\Gamma;m_1,\ldots,m_n)$, where $m_i \in \bbZ_{\geq0}$, as the disjoint union $\Lambda(\Gamma) \cupdot \CP_{m_1} \cupdot \cdots \cupdot \CP_{m_n}$ with additional edges edges from every vertex in $\CP_{m_i}$ to every $v_iv_j \in V(\Lambda(\Gamma))$.  (It is the line graph $\Lambda(\Gamma)$ if all $m_i=0$.)  
Hoffman showed that a generalized line graph has least eigenvalue $\geq -2$, just like a line graph.  

Here is $\Lambda(C_4; 1,2,0,0)$:\qquad
\raisebox{-2cm}{\includegraphics[scale=.8]{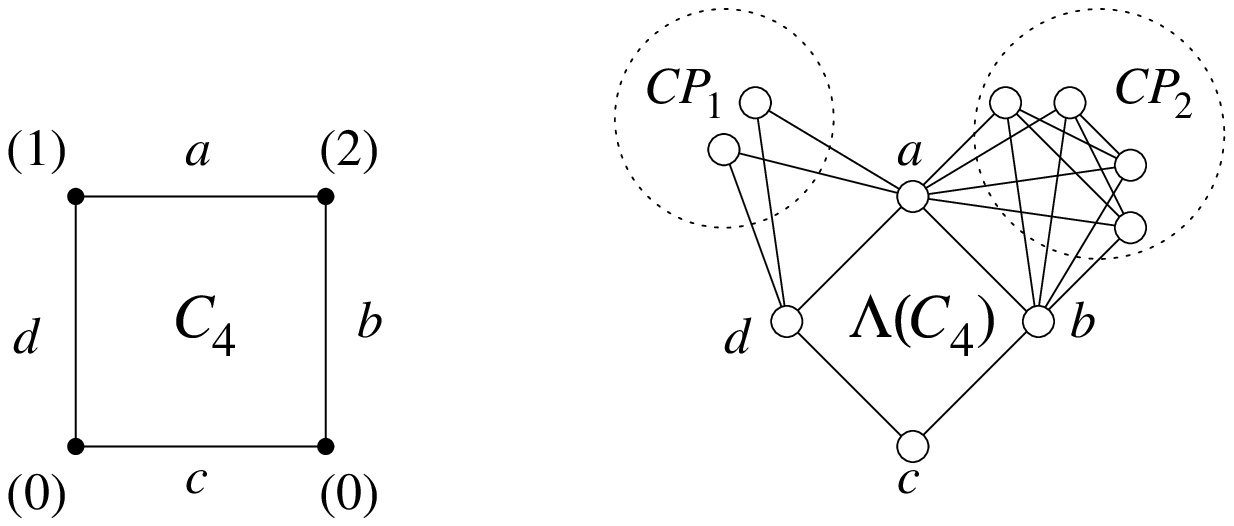}}
\bigskip

Hoffman's eigenvalue theorem is an easy consequence of Corollary \ref{T:lgebound}.  We deduce it by showing how $\Lambda(\Gamma;m_1,\ldots,m_n)$ is a reduced line graph of a signed graph.  Let $\Gamma(m_1,\ldots,m_n)$ be $\Gamma$ with $m_i$ negative digons attached to $v_i$.  The other vertex of each negative digon is a new vertex; thus, $\Gamma(m_1,\ldots,m_n)$ has order $n+m_1+\cdots+m_n$ and $|E|+2(m_1+\cdots+m_n)$ edges.  Then $-\Gamma(m_1,\ldots,m_n)$ is $-\Gamma$ with the negative digons adoined (since a negated negative digon is still a negative digon), and $\bar\Lambda(-\Gamma(m_1,\ldots,m_n)) = - \Lambda(\Gamma;m_1,\ldots,m_n)$.  The eigenvalue property of $\Lambda(\Gamma;m_1,\ldots,m_n)$ follows immediately from Theorem \ref{T:grame}.

Here is the construction of $-\Lambda(C_4;1,2,0,0)$ as the reduced signed line graph $\bar\Lambda(-C_4(1,2,0,0))$.  It begins with all negative edges extraverted in  $-\Lambda(C_4;1,2,0,0)$:\\

\begin{center}
\parbox{2in}{
\begin{center}
$-C_4(1,2,0,0)$
\\[8pt]
{\includegraphics[scale=.8]{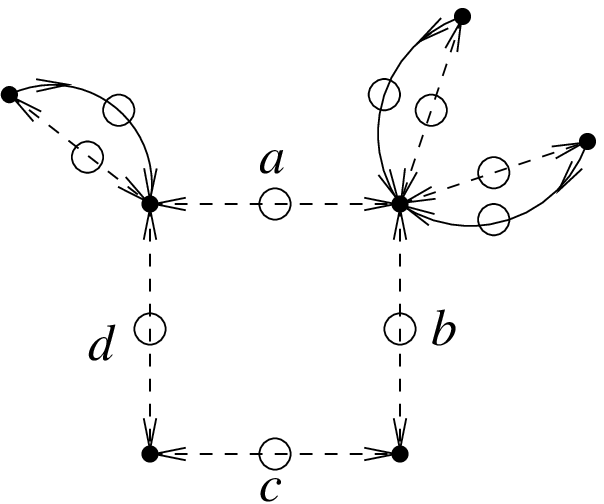}}
\end{center}
}
\qquad
\parbox{2.7in}{
\begin{center}
$-\Lambda(C_4;1,2,0,0) = \bar\Lambda(-C_4(1,2,0,0))$
\\[8pt]
{\includegraphics[scale=.8]{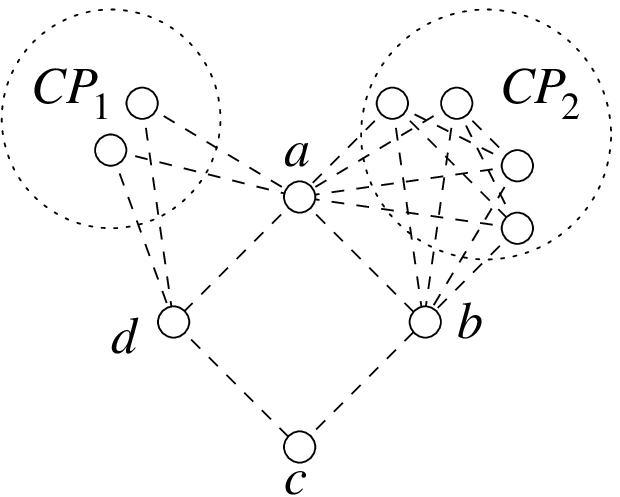}}
\end{center}
}
\end{center}
\end{exam}

%%===================================================%%

\section{Acknowledgements}\label{ack}

This paper is based on lecture notes from the International Workshop on Set-Valuations, Signed Graphs, Geometry and Their Applications (IWSSG-2011), Mananthavady, Kerala, 2--6 September 2011, which in turn were enlarged from lecture notes of  the 2010 workshop at Pala and Mananthavady.  For that delightful 2010 workshop I thank Professor Sr.\ Germina K.A.\ and Professor A.M.\ Mathai.

%%===================================================%%

\renewcommand\refname{Keys to the Literature}

\renewcommand\refname{Additional Bibliography}

\end{document}